\documentclass[12pt]{amsart}
\usepackage{amscd,amssymb,float}
\usepackage{amscd}
\usepackage{amssymb}
\usepackage{graphicx}
\usepackage{color}
\usepackage[centering,text={15.5cm,22cm},
        marginparwidth=20mm]{geometry}
\usepackage{mathrsfs}

\usepackage[all]{xy}

\newtheorem{theorem}{Theorem}[section]
\newtheorem{lemma}[theorem]{Lemma}

\newtheorem{corollary}[theorem]{Corollary}
\newtheorem{proposition}[theorem]{Proposition}

\theoremstyle{definition}

\theoremstyle{remark}
\newtheorem{remark}[theorem]{Remark}
\numberwithin{equation}{section}

\DeclareMathOperator{\Vol}{Vol}

 \DeclareMathOperator{\leng}{length}

\begin{document}
\title{Continuity of  Extremal  Transitions  and Flops for Calabi-Yau Manifolds}
\author[X.Rong]{Xiaochun Rong * }
\thanks{*Supported partially  by NSF Grant DMS-0805928, and by research found from Capital Normal
University.  \\ Address: Mathematics Department, Capital Normal
University, Beijing 100048, P.R.China, and  Mathematics Department,
Rutgers University New Brunswick, NJ 08903, USA.  E-mail address:
rong@math.rutgers.edu}
\author[Y. Zhang]{Yuguang Zhang **}
\thanks{**Supported by  NSFC-10901111, and KM-210100028003.  \\ Address:   Mathematics Department, Capital Normal
University, Beijing 100048, P.R.China.   E-mail address:
 yuguangzhang76@yahoo.com}

\begin{abstract} In this paper, we study the behavior of Ricci-flat
K\"{a}hler metrics  on Calabi-Yau manifolds under algebraic
geometric surgeries: extremal transitions or  flops.
 We prove a version of Candelas and de la Ossa's conjecture:   Ricci-flat
 Calabi-Yau  manifolds related  by
extremal transitions and flops can be connected by a  path
consisting of   continuous  families  of Ricci-flat
 Calabi-Yau manifolds and a compact metric space  in the  Gromov-Hausdorff
 topology. In an essential step of the  proof of our main result,
 the convergence of Ricci-flat
K\"{a}hler metrics  on Calabi-Yau manifolds along a smoothing is
established, which can be of  independent interests.
\end{abstract}

\maketitle
 \markboth{Continuity of  Extremal  Transitions  and Flops for Calabi-Yau
 Manifolds}{Continuity of  Extremal  Transitions  and Flops for Calabi-Yau
 Manifolds}

\section{Introduction }
\vskip2mm

A Calabi-Yau manifold $M$ is a simply connected   projective
manifold with trivial canonical bundle $ \mathcal{K}_{M}\cong
\mathcal{O}_{M}$. In the 1970's, S.T.Yau  proved Calabi's conjecture
in  \cite{Ya1}, which says that, for
    any K\"{a}hler class $\alpha\in H^{1,1}(M, \mathbb{R})$,
there exists a
     unique Ricci-flat K\"{a}hler metric $g$ on $M$ with  K\"{a}hler form
     $\omega\in\alpha$. The study of Calabi-Yau manifolds became
     very interesting in the last three decades (cf.  \cite{Y4}).
     The convergence of Ricci-flat Calabi-Yau manifolds was studied
     from various perspectives  (cf. \cite{An1} \cite{Cha} \cite{C} \cite{CT}
     \cite{GW2}  \cite{Kob1}  \cite{Lu1}
      \cite{To} \cite{To2}   \cite{To3} \cite{RZ}   \cite{Wi} \cite{ZY0}).
     The  goal of the present paper is to study the metric
     behavior of Calabi-Yau manifolds   under
some algebraic geometric  surgeries.

   Let $M_{0}$ be a singular  projective normal
variety with singular set $S$. Usually there are two type of
desingularizations: one is a
 {\it resolution} $(\bar{M},\bar{\pi})$, i.e., $\bar{M}$ is a
 projective manifold, and $\bar{\pi}$ is a morphism such that $\bar{\pi}:
  \bar{M}\backslash \bar{\pi}^{-1}(S)\rightarrow M_{0} \backslash S$
  is bi-holomorphic.
The other is a {\it smoothing} $(\mathcal{M}, \pi)$ over the unit disc
$\Delta\subset \mathbb{C}$, i.e.,   $\mathcal{M}$ is an
$(n+1)$-dimensional variety, $\pi$ is a proper flat morphism,
$M_{0}=\pi^{-1}(0)$, and $M_{t}=\pi^{-1}(t)$ is a smooth projective
$n$-dimensional manifold for any $t\in \Delta\backslash \{0\}$.  If $M_{0}$
admits a resolution
 $ (\bar{M}, \bar{\pi})$ and  a smoothing   $(\mathcal{M}, \pi)$,    the process of
going from $\bar{M}$ to $M_{t}$, $t\neq 0$,  is called an  {\it
extremal transition}, denoted by $\bar{M}\rightarrow M_{0}
\rightsquigarrow M_{t}$. We call this process a {\it conifold
transition} if $M_{0}$ is a conifold, which
 is    a normal variety  $M_{0}$  with  only finite  ordinary double points as
singularities, i.e.,  any singular point is locally  given by
$$z_{0}^{2}+\cdots +z_{n}^{2}=0, \ \ {\rm where } \ \ \dim_{\mathbb{C}}
M_{0}=n.$$
 If $M_{0}$ admits two different
 resolutions $ (\bar{M}_{1}, \bar{\pi}_{1})$ and
 $(\bar{M}_{2}, \bar{\pi}_{2})$ with both exceptional subvarieties of codimension  at least  2, the process of
going from $\bar{M}_{1}$ to  $\bar{M}_{2}$ is called  a \emph{flop},
denoted by $\bar{M}_{1} \rightarrow M_{0} \dashrightarrow
\bar{M}_{2}$.

Extremal  transitions and flops  are algebraic geometric
 surgeries
  providing   ways to connect two  topologically distinct projective
manifolds, which is interesting  in  both mathematics and physics.
In the  minimal model program, all smooth   minimal models of
dimension 3 in a birational equivalence class are connected by a
sequence of flops (cf. \cite{Kol} \cite{KM}).  The famous Reid's
fantasy
 conjectures  that all Calabi-Yau threefolds are connected to
each other by   extremal transitions, possibly including non-K\"{a}hler
Calabi-Yau threefolds,   so as to form a huge connected web (cf.
\cite{Re} \cite{Ro}).  There is also  a  projective version of this
conjecture,  the connectedness conjecture for  moduli spaces for
Calabi-Yau threefolds (cf. \cite{Gro} \cite{Gro0} \cite{Ro}).
Furthermore, in physics, flops and extremal transitions are related
to the topological change of
 the  space-time in string theory (cf. \cite{CGH} \cite{CGGK} \cite{GMS} \cite{GH}
  \cite{Ro}). Readers are referred to  the
survey article \cite{Ro} for topology, algebraic geometry, and even
physics
   properties of
 extremal    transitions.

 In \cite{Ca},    physicists   P.Candelas and X.C.de la Ossa
   conjectured   that
extremal  transitions and flops  should be
 ``continuous in the space of Ricci-flat K\"{a}hler  metrics", even though   these processes  involve  topologically
distinct Calabi-Yau manifolds. This conjecture was verified  in
\cite{Ca}  for the non-compact quadric cone $M_{0}=\{(z_{0}, \cdots
, z_{3})\in \mathbb{C}^{4}|z^{2}_{0}+\cdots +z^{2}_{3}=0\}$.

 In  the 1980's, Gromov introduced the notion of  Gromov-Hausdorff
distance $d_{GH} $ on the space $\mathfrak{X} $ of isometric classes
of all compact metric spaces (cf. \cite{G1}), such that
$(\mathfrak{X},d_{GH})$ is a complete metric space (cf. \cite{G1}
and Appendix A).  This notion provides a framework to study the
continuity of a family of compact metric spaces with possibly
different topologies. The Gromov-Hausdorff topology provides  a
natural mathematical formulation of Candelas and de la Ossa's
conjecture  as follows:
\begin{itemize}
  \item[i)]   If $\bar{M}\rightarrow M_{0} \rightsquigarrow
M_{t}$, $t\in \Delta\backslash\{0\}\subset \mathbb{C}$, is an
extremal transition  among  Calabi-Yau manifolds, then there exists
a family of Ricci-flat K\"{a}hler metrics $\bar{g}_{s}$, $s\in
(0,1)$, on $\bar{M}$, and  a family of Ricci-flat K\"{a}hler metrics
$\tilde{g}_{t}$  on $M_{t}$  satisfying that $\{( \bar{M},
\bar{g}_{s})\}$ and $\{(M_{t}, \tilde{g}_{t})\}$ converge to a
single
   compact metric space $(X, d_{X})$ in the Gromov-Hausdorff
topology,   $$(M_{t}, \tilde{g}_{t})
\stackrel{d_{GH}}\longrightarrow (X, d_{X})
\stackrel{d_{GH}}\longleftarrow (\bar{M}, \bar{g}_{s}), \ \ \ \ \
 \ \  s\rightarrow 0, \ t\rightarrow0.$$  \item[ii)]  If $\bar{M}_{1} \rightarrow M_{0} \dashrightarrow
\bar{M}_{2}$ is a flop between two Calabi-Yau manifolds, then  there
are families  of Ricci-flat K\"{a}hler metrics $\bar{g}_{i,s}$,
$s\in (0,1)$ on   $\bar{M}_{i}$ ($i=1,2$)  such that
$$( \bar{M}_{1}, \bar{g}_{1,s}) \stackrel{d_{GH}}\longrightarrow (X,
d_{X}) \stackrel{d_{GH}}\longleftarrow ( \bar{M}_{2},
\bar{g}_{2,s}),   \ \ \ \ \ \ \ \  s\rightarrow 0,
$$ for a single  compact
metric space  $(X, d_{X})$.
\end{itemize}
  Furthermore,  in both cases  $X$  is homeomorphic to $M_{0}$ and $d_{X}$ is induced by a Ricci-flat K\"{a}hler
metric on  $M_{0}\backslash S$.
   In  the present   paper, we shall  prove  i) and ii) of   the above version of  Candelas and de la Ossa's conjecture.

 Let  $M_{0}$ be a projective  normal  Cohen-Macaulay
$n$-dimensional variety  with singular set $S$, and let
$\mathcal{K}_{M_{0}}$ be the canonical sheaf of
 $M_{0}$ (\cite{Ha}).  In this paper, all varieties are assumed to be
 Cohen-Macaulay.  We call
  $ M_{0}$  Gorenstein if   $\mathcal{K}_{M_{0}}$ is a rank one locally free
  sheaf.     Assume that $M_{0}$  has  only canonical
singularities, i.e.,  $M_{0}$ is Gorenstein, and   for any resolution
$( \bar{M}, \bar{\pi})$,
 $$\mathcal{K}_{\bar{M}}=\bar{\pi}^{*}\mathcal{K}_{M_{0}}+\sum a_{E}E,  \quad a_{E}\geq
 0,$$ where $E$ are effective  exceptional divisors.
   Consider a resolution $(\bar{M}, \bar{\pi})$ of $M_{0}$.
  If  $ \alpha $ is an ample class in the Picard group
of  $ M_{0}$,
 $\bar{\pi}^{*} \alpha$ belongs  the boundary of the K\"{a}hler cone
of $\bar{M}$.   A resolution $(\bar{M}, \bar{\pi})$ of $M_{0}$ is
called a \emph{crepant resolution}  if
$\mathcal{K}_{\bar{M}}=\bar{\pi}^{*}\mathcal{K}_{M_{0}} $  and is
called a \emph{small resolution} if the exceptional subvariety
$\bar{\pi}^{-1} (S)$ satisfies  $
\dim_{\mathbb{C}}\bar{\pi}^{-1} (S)\leq n-2$.  It is obvious that
$(\bar{M}, \bar{\pi})$ is crepant if it is a small resolution.
  If
$M_{0}$ admits a smoothing $(\mathcal{M}, \pi)$ over a unit  disc $
\Delta\subset\mathbb{C}$ with  an ample line bundle $\mathcal{L}$ on
$ \mathcal{M}$, then  there is an embedding
$\mathcal{M}\hookrightarrow \mathbb{CP}^{N}\times \Delta$ such that
$ \mathcal{L}^{m}=\mathcal{O}_{\Delta}(1)|_{\mathcal{M}}$ for some
$m\geq 1$,  $\pi$ is a proper surjection given by
   the restriction of
 the projection from $\mathbb{CP}^{N}\times \Delta $ to $\Delta$, and
     the rank of $\pi_{*}$ is 1 on $
\mathcal{M}\backslash S$. This implies that $M_{t}$, $t\in
\Delta\backslash\{0\}$, have the same underlying   differential
manifold $ \tilde{M}$.  Moreover, if $\mathcal{L}$ is a line bundle
on $\mathcal{M}$ such that the restriction of  $\mathcal{L}$  on
$M_{0}$  is ample, then by Proposition 1.41 in \cite{KM}
$\mathcal{L}$ is ample on $\pi^{-1}(\Delta')$ where
$\Delta'\subseteq \Delta$ is a neighborhood of $0$.

   A Calabi-Yau variety is a simply connected
  projective normal  variety $M_{0}$ with  trivial canonical
sheaf $\mathcal{K}_{M_{0}}\cong \mathcal{O}_{M_{0}}$ and only
canonical singularities.  If a Calabi-Yau variety $M_{0}$ admits a
crepant resolution $(\bar{M}, \bar{\pi})$, then $\bar{M}$ is a
Calabi-Yau manifold. Our first result proves i) in the   above
version of Candelas and de la Ossa's conjecture.

 \vskip 3mm

\begin{theorem}\label{t1.1}Let  $M_{0}$ be a  Calabi-Yau
$n$-variety  with singular set $S$. Assume that
\begin{itemize}
  \item[i)] $M_{0}$ admits a smoothing $\pi:
\mathcal{M}\rightarrow \Delta$  over  the unit disc $\Delta\subset
\mathbb{C}$  such that the relative   canonical bundle
$\mathcal{K}_{\mathcal{M}/\Delta}$ is trivial, i.e.,
$\mathcal{K}_{\mathcal{M}/\Delta}\cong \mathcal{O}_{\mathcal{M}} $
and  $ \mathcal{M}$ admits    an ample line bundle $\mathcal{L}$.
For any $t\in \Delta\backslash \{0\}$, let $\tilde{g}_{t}$ be the
unique Ricci-flat
  K\"{a}hler metric on $M_{t}=\pi^{-1}(t) $ with  K\"{a}hler form $\tilde{\omega}_{t}\in  c_{1}(\mathcal{L})|_{M_{t}}
  $.\item[ii)] $M_{0}$ admits a crepant  resolution $(\bar{M}, \bar{\pi})$.
  Let
  $\{\bar{g}_{s}\}$ ($s\in (0, 1] $) be  a family of Ricci-flat
  K\"{a}hler metrics with  K\"{a}hler classes $\lim\limits_{s\rightarrow 0}[\bar{\omega}_{s}]=\bar{\pi}^{*} c_{1}(\mathcal{L})|_{M_{0}} $
   in $H^{1,1}(\bar{M}, \mathbb{R} )$, where $\bar{\omega}_{s}$ denotes
   the corresponding K\"{a}hler form of  $ \bar{g}_{s}$.
   \end{itemize}
  Then there exists a compact length metric space $(X, d_{X})$ such
  that $$ \lim_{t\rightarrow 0}d_{GH}((M_{t}, \tilde{g}_{t}),(X, d_{X}))=
  \lim_{s\rightarrow 0}d_{GH}((\bar{M}, \bar{g}_{s}),(X, d_{X}))=0. $$
   Furthermore,  $(X, d_{X})$ is isometric to the metric completion
   $\overline{(M_{0}\backslash S,d_{g})} $ where $g$ is a Ricci-flat
  K\"{a}hler metric on $M_{0}\backslash S$, and $d_{g}$ is
  Riemannian distance function of $g$.
 \end{theorem}

 \vskip 3mm

The following is a simple example from \cite{GMS} for which  Theorem
\ref{t1.1} can apply. Let $\bar{M}$ be the complete intersection in
$\mathbb{CP}^{4}\times \mathbb{CP}^{1} $ given by
$$y_{0}\mathfrak{g}(z_{0},\cdots ,z_{4})+y_{1}\mathfrak{h}(z_{0},\cdots,z_{4})=0,
\ \ \ y_{0}z_{4} -y_{1}z_{3}=0, $$ where $z_{0},\cdots,z_{4}$ are
homogeneous coordinates of $\mathbb{CP}^{4}$, $y_{0},y_{1}$ are
homogeneous coordinates of $\mathbb{CP}^{1}$, and $\mathfrak{g}$ and
$\mathfrak{h}$ are generic  homogeneous polynomials of degree 4.
Then $\bar{M}$ is a  crepant resolution of the quintic conifold
$M_{0}$ given by $$z_{3}\mathfrak{g}(z_{0},\cdots
,z_{4})+z_{4}\mathfrak{h}(z_{0},\cdots,z_{4})=0$$  (cf. \cite{Ro}).
Hence there is a conifold transition $\bar{M}\rightarrow M_{0}
\rightsquigarrow \tilde{M} $ for any smooth quintic $\tilde{M}$ in
$\mathbb{CP}^{4}$.  Theorem  \ref{t1.1} implies that there is a
family of Ricci-flat K\"{a}hler metrics
 $\bar{g}_{s}$ ($s\in (0, 1] $) on $\bar{M} $ and a family of Ricci-flat  smooth quintic
 $(M_{t},\tilde{g}_{t})$ ($ t\in\Delta\backslash \{0\}$)   such that
 $M_{1} =\tilde{M}$, and $$(M_{t}, \tilde{g}_{t}) \stackrel{d_{GH}}\longrightarrow
(X, d_{X}) \stackrel{d_{GH}}\longleftarrow  (\bar{M}, \bar{g}_{s}),
$$ for a compact metric space $(X, d_{X})$.

Our second  result proves ii) in the   above version of Candelas and de
la Ossa's conjecture.

\vskip 3mm

\begin{theorem}\label{t1.01+} Let  $M_{0}$ be an $n$-dimensional Calabi-Yau
variety  with singular set $S$,  and $\mathcal{L}$ be an ample
line bundle.  Assume that $M_{0}$ admits two crepant resolutions
$(\bar{M}_{1}, \bar{\pi}_{1})$ and $(\bar{M}_{2}, \bar{\pi}_{2})$.
 Let
 $\{\bar{g}_{1,s}\}$ (resp. $\{\bar{g}_{2,s}\}$ $s\in (0, 1] $)  be a family of Ricci-flat
  K\"{a}hler metrics on $\bar{M}_{1} $ (resp. $\bar{M}_{2} $) with  K\"{a}hler classes
  $\lim\limits_{s\rightarrow 0}[\bar{\omega}_{\alpha,s}]=\bar{\pi}_{\alpha}^{*} c_{1}(\mathcal{L})
  $, $\alpha=1,2$.
  Then there exists a compact length metric space $(X, d_{X})$ such
  that $$ \lim_{s\rightarrow 0}d_{GH}((\bar{M}_{1}, \bar{g}_{1,s}),(X, d_{X}))=
  \lim_{s\rightarrow 0}d_{GH}((\bar{M}_{2}, \bar{g}_{2,s}),(X, d_{X}))=0. $$
   Furthermore, $(X, d_{X})$ is isometric to the metric completion
   $\overline{(M_{0}\backslash S,d_{g})} $ where $g$ is a Ricci-flat
  K\"{a}hler metric on $M_{0}\backslash S$, and $d_{g}$ is
  the Riemannian distance function of $g$.
\end{theorem}

 \vskip 3mm

 \begin{remark} The present arguments are inadequate to  prove that  $X$ is homeomorphic to
$M_{0}$ in both Theorem \ref{t1.1} and Theorem \ref{t1.01+}.
Additional work is required. However, if $M_{0}$ has only orbifold
singularities, and $c_{1}(\mathcal{L}) $ can be represented by an
orbifold K\"{a}hler metric on $M_{0}$, then $X$ is homeomorphic to
$M_{0}$ by
 Corollary 1.1 in \cite{RZ}.
\end{remark}
 \vskip 3mm

 We now begin to describe our approach to Theorem \ref{t1.1}  and Theorem  \ref{t1.01+}.
  Let  $M_{0}$ be a normal $n$-dimensional projective variety with singular set
$S$.
  For any $p\in S$ and a  small neighborhood $U_{p}\subset M_{0}$ of $p$,
a pluri-subharmonic function $v$ (resp. strongly pluri-subharmonic,
and pluri-harmonic) on $U_{p}$ is an upper semi-continuous function
with value in $\mathbb{R}\cup \{-\infty\}$ ($v$  is not locally
$-\infty$)  such that  $v$ extends to a pluri-subharmonic function $
\tilde{v}$ (resp. strongly pluri-subharmonic, and pluri-harmonic) on
a neighborhood of the image of  some local embedding $U_{p}
\hookrightarrow \mathbb{C}^{m}$.  We call $v$ smooth if  $
 \tilde{v}$ is smooth.
     A   form $\omega$  on $M_{0}$
       is called a K\"{a}hler form, if $\omega$  is a smooth K\"{a}hler form  in the usual
sense on $M_{0}\backslash S $ and, for any  $p\in S$, there is a
neighborhood $U_{p}$  and a continuous strongly pluri-subharmonic
function $v$ on
 $U_{p}$ such that
 $\omega=\sqrt{-1}\partial\overline{\partial}v$ on $U_{p}\bigcap (M_{0}\backslash
 S)
 $. We call $\omega$  smooth if $ v$ is
 smooth in the  above sense. Otherwise, we call $\omega$ a singular K\"{a}hler form.
 If $\mathcal{PH}_{M_{0}}$ denotes the sheaf of pluri-harmonic
functions on $M_{0}$, then  any K\"{a}hler form $\omega$ represents
a class $[\omega]$ in $H^{1}(M_{0}, \mathcal{PH}_{M_{0}})$ (cf.
Section 5.2 in \cite{EGZ}).  We also have an  analogue of Chern-Weil
theory for line bundles on $M_{0}$ (see \cite{EGZ} for details).
 If $ \mathcal{L}_{0}$ is an ample line bundle on $M_{0}$, then there is an embedding $M_{0}
\hookrightarrow \mathbb{CP}^{N}$ such   that $
\mathcal{L}_{0}^{m}=\mathcal{O}(1)|_{M_{0}}$, and the first Chern
class $c_{1}( \mathcal{L}_{0}) $ can be presented by a smooth
K\"{a}hler form;   $ c_{1}( \mathcal{L}_{0})=\frac{1}{
m}[\omega_{FS}|_{M_{0}}]\in  H^{1}(M_{0}, \mathcal{PH}_{M_{0}})$,
where $\omega_{FS}$ denotes the standard Fubini-Study K\"{a}hler
form on $\mathbb{CP}^{N}$.

In \cite{EGZ} (see also \cite{Z}), a generalized Calabi-Yau theorem
was obtained, which says that if $M_{0}$ is a Calabi-Yau variety,
then  for any ample line bundle $ \mathcal{L}_{0}$ there is a unique
Ricci-flat K\"{a}hler  form $\omega \in c_{1}( \mathcal{L}_{0})$. We
denote by $g$ the corresponding K\"{a}hler metric of $\omega$ on
$M_{0}\backslash S$.   If  $M_{0}$ admits a crepant resolution
$(\bar{M}, \bar{\pi})$, and $\alpha_{s}\in H^{1,1}(\bar{M},
\mathbb{R})$, $ s\in (0,1)$, is a family of K\"{a}hler classes with
$\lim\limits_{s\rightarrow 0} \alpha_{s}=\bar{\pi}^{*} c_{1}(
\mathcal{L}_{0})$, \cite{To} proved that
$$\bar{g}_{s} \longrightarrow \bar{\pi}^{*}g, \ \ \ \bar{\omega}_{s}
\longrightarrow \bar{\pi}^{*}\omega, \ \ \ \ \  s\rightarrow 0$$ in
the $C^{\infty}$-sense on any compact subset $K$  of
$\bar{M}\backslash \bar{\pi}^{-1}(S)$, where $\bar{g}_{s}$ is the
unique Ricci-flat K\"{a}hler metric with K\"{a}hler form
$\bar{\omega}_{s}\in \alpha_{s}$.  Assume that $M_{0}$ is a
Calabi-Yau conifold  and $M_{0}$  admits a smoothing $
(\mathcal{M},\pi)$ satisfying  that the relative canonical bundle
$\mathcal{K}_{\mathcal{M}/\Delta}$ is trivial  and that  $
\mathcal{M}$
   admits   an ample line bundle $\mathcal{L}$ such
that $\mathcal{L}|_{M_{0}}=\mathcal{L}_{0}$. For any $t\in
\Delta\backslash \{0\}$, if $\tilde{g}_{t}$ denotes the unique
Ricci-flat
  K\"{a}hler metric on $M_{t}=\pi^{-1}(t) $ with  K\"{a}hler form $\tilde{\omega}_{t}\in
c_{1}(\mathcal{L})|_{M_{t}}$,  \cite{RZ}  proved that $$
F_{t}^{*}\tilde{g}_{t} \longrightarrow  g, \ \
F_{t}^{*}\tilde{\omega}_{t} \longrightarrow  \omega, \  \ \ {\rm
when} \ \
 t \rightarrow0,$$
 in the $C^{\infty}$-sense on  any compact subset $K \subset
M_{0}\backslash S$,  where $F_{t}: M_{0}\backslash S \longrightarrow
M_{t}$ is a family of embeddings. If $M_{0}$ is  a Calabi-Yau
variety (not necessarily  a  conifold)  and   $ \mathcal{M}$ is
smooth, then a subsequence-$C^{\infty}$ convergence theorem for the
Ricci-flat K\"{a}hler metric $\tilde{g}_{t}$ on $M_{t}$ was obtained
in \cite{RZ}, i.e., there is a sequence $t_{k}\in \Delta\backslash
\{0\}$ such that $t_{k}\rightarrow 0$, and  $
F_{t_{k}}^{*}\tilde{g}_{t_{k}} $ converges to $  g$ ($
k\rightarrow\infty$)  in the $C^{\infty}$-sense on any compact
subset $K \subset M_{0}\backslash S$.

In the proof of Theorem 1.1, the following  generalization of the
  convergence  results in \cite{RZ} plays a  significant role.

\vskip 3mm
\begin{theorem}\label{t4.2+}
 Let  $M_{0}$ be a Calabi-Yau
$n$-variety ($n\geq 2$)   with singular set $S$. Assume that $M_{0}$
admits a smoothing $\pi: \mathcal{M}\rightarrow \Delta$   such that
$\mathcal{M}$
 admits  an ample line bundle $\mathcal{L}$   and   the
relative canonical bundle is trivial, i.e.,
$\mathcal{K}_{\mathcal{M}/\Delta}\cong \mathcal{O}_{\mathcal{M}}$.
If $\tilde{g}_{t}$ denotes the unique Ricci-flat  K\"{a}hler  metric
with K\"{a}hler form $\tilde{\omega}_{t}\in
c_{1}(\mathcal{L})|_{M_{t}}\in H^{1,1}(M_{t}, \mathbb{R})$  ($t\in
\Delta\backslash \{0\}$), and   $\omega$ denotes   the unique
singular Ricci-flat K\"{a}hler form  on $M_{0}$ with    $\omega\in
c_{1}(\mathcal{L})|_{M_{0}} \in H^{1}(M_{0}, \mathcal{PH}_{M_{0}})$,
 then
$$ F_{t}^{*}\tilde{g}_{t} \longrightarrow  g, \ \ F_{t}^{*}\tilde{\omega}_{t} \longrightarrow  \omega, \  \ \ {\rm
when} \ \
 t \rightarrow0,$$
 in the $C^{\infty}$-sense on  any compact subset $K \subset
M_{0}\backslash S$,  where $F_{t}: M_{0}\backslash S \longrightarrow
M_{t}$ is a smooth  family of embeddings  and $g$ is the
corresponding K\"{a}hler metric of $\omega$ on $M_{0}\backslash S$.
Furthermore, the diameter of $ (M_{t}, \tilde{g}_{t})$ ($t\in
\Delta\backslash \{0\}$)  satisfies $${\rm
diam}_{\tilde{g}_{t}}(M_{t})\leq D,$$ where $D>0$ is a constant
independent of $t$.
\end{theorem}
\vskip 3mm

Our  proof of Theorem \ref{t1.1} is to show that $(M_{0}\backslash
S, g) $ has a  metric completion $(X, d_{X})$ satisfying the property
that both
$\{( M, \bar{g}_{s})\}$ and $\{(M_{t}, \tilde{g}_{t})\}$ converge to
$(X, d_{X})$ in the Gromov-Hausdorff topology  when $s\rightarrow 0$
and $t\rightarrow0$. The same method also proves  Theorem
\ref{t1.01+}.

As an application of Theorem \ref{t1.1} and Theorem \ref{t1.01+}, we
shall explore the path connectedness
 properties of certain class of  Ricci-flat Calabi-Yau threefolds.
Inspired by string theory in physics, some physicists made a
 projective version  of Reid's fantasy (cf.  \cite{CGH} \cite{GH}
 and
\cite{Ro}), the so-called connectedness conjecture,  which is
formulated more precisely  in \cite{Gro} (See also \cite{Gro0}).
This conjecture says that there is a huge connected  web $ \Gamma$
such that   nodes of $ \Gamma$ consist of   all deformation classes
of Calabi-Yau threefolds, and two nodes are connected
$\mathfrak{D}_{1}- \mathfrak{D}_{2}$ if $\mathfrak{D}_{1}$ and $
\mathfrak{D}_{2}$ are related by an extremal transition, i.e.,  there
is a Calabi-Yau 3-variety $M_{0}$ that  admits a crepant resolution
$\bar{M}\in \mathfrak{D}_{1}$ and a smoothing $(\mathcal{M}, \pi)$
  satisfying
$\pi^{-1}(t)=M_{t}\in \mathfrak{D}_{2}$ for any $ t\in
\Delta\backslash\{0\}$.  It was shown in \cite{GH}, \cite{CGGK},
\cite{ACJM} and \cite{Gro} that many Calabi-Yau threefolds are
connected to each other in the above sense.
 By combining the connectedness conjecture
and  Theorem \ref{t1.1} and Theorem \ref{t1.01+}, we reach a metric
version of connectedness conjecture as follows: if $
\mathfrak{CY}_{3}$ denotes  the set of Ricci-flat Calabi-Yau
threefolds $(M,g) $ with volume 1, then  the closure $
\overline{\mathfrak{CY}}_{3}$ of $ \mathfrak{CY}_{3}$ in
$(\mathfrak{X},d_{GH}) $ is path connected, i.e.,  for any two
points  $p_{1}$ and $p_{2}\in \overline{\mathfrak{CY}}_{3}$, there
is a  path
$$\gamma:[0,1] \longrightarrow \overline{\mathfrak{CY}}_{3}\subset
(\mathfrak{X},d_{GH})$$ such that $ p_{1}=\gamma(0)$ and
$p_{2}=\gamma(1)$.

Given  a class   of Calabi-Yau 3-manifolds known  to be connected by
extremal transitions and flops in algebraic geometry, Theorem
\ref{t1.1} and Theorem \ref{t1.01+} can be used to show that the
closure of the  class of  Calabi-Yau 3-manifolds
   is path connected in $(\mathfrak{X},d_{GH}) $. In the minimal
   model program, it was proved  that  for any two Calabi-Yau
   3-manifolds $M$ and $M'$  birational to each other, there is a sequence of
   flops connecting  $M$ and $M'$ (cf.  \cite{Kol}  \cite{KM}).  In
   \cite{GH}, it was shown  that all complete  intersection  Calabi-Yau
   manifolds  (CICY) of dimension 3 in  products  of
  projective
 spaces  are connected by conifold  transitions. Furthermore,
   in  \cite{ACJM} and  \cite{CGGK}   a large number of complete
intersection  Calabi-Yau
   3-manifolds in toric varieties  were verified  to be
   connected
   by extremal transitions, which include
   Calabi-Yau hypersurfaces in all  toric manifolds obtained by resolving
     weighted projective 4-spaces.   As a corollary  of Theorem \ref{t1.1} and Theorem \ref{t1.01+},  we
 obtain the following result.

 \vskip 3mm
 \begin{corollary}\label{c1.2} For any Calabi-Yau manifold $M$, let  $$\mathfrak{M}_{M}=\{(M,g)\in\mathfrak{X}| \  g \  is \ a \ Ricci-flat \
  K\ddot{a}hler  \  metric \  on  \ M \  with  \  {\rm Vol}_{g}(M)=1
  \}.
  $$ \begin{itemize}
  \item[i)] If $M$ is a three-dimensional Calabi-Yau manifold, and
   $$\mathfrak{BM}_{M}=\bigcup_{
  all \ Calabi-Yau \ manifolds \ M' \  birational \ to \ M } \mathfrak{M}_{M'}
 ,$$ then  the closure $\overline{\mathfrak{BM}}_{M}$ of
 $\mathfrak{BM}_{M}$ in $(\mathfrak{X}, d_{GH}) $ is path connected.
  \item[ii)]  Let $$\mathfrak{CP}=\bigcup_{
  all \ CICY \ 3-manifolds   \  M'  \ in \ products \ of \
  projective \
 spaces
 } \mathfrak{M}_{M'}
 .$$ Then  the closure $\overline{\mathfrak{CP}}$ of
 $\mathfrak{CP}$ in $(\mathfrak{X}, d_{GH}) $ is path connected.
 \item[iii)]  There is a path connected component  $ \overline{\mathfrak{CT}}$ of $ \overline{\mathfrak{CY}}_{3}\subset
 (\mathfrak{X},d_{GH})$ such that $\mathfrak{CP}\subset \overline{\mathfrak{CT}}$, and
 $\overline{\mathfrak{CT}}$ contains all $(M,g)$,  where  $M$ is a   Calabi-Yau hypersurface  in
 a
 toric 4-manifold   obtained by resolving a  weighted projective
 4-space, and  $g$ is a Ricci-flat K\"ahler metric of volume 1 on $M$.
 \end{itemize}
 \end{corollary}

 \vskip 3mm

 The study of metric behaviors  under some algebraic geometric surgeries also
arises from
 other perspectives,  such as  K\"ahler-Ricci flow  (cf. \cite{ST}  \cite{SW0} \cite{SW} and \cite{SY}) and balanced
 metrics  on non-K\"ahler Calabi-Yau  threefolds   (cf. \cite{FLY} ).

The rest of the paper is organized as follows: In Section 2, we
bound from above of diameters of Ricci-flat Calabi-Yau manifolds
along a  smoothing.  In Section 3, we prove Theorem \ref{t4.2+}.
 In Section 4, we establish a link between   point-wise
 $C^{\infty}$-convergence of Riemannian  metrics  on a  `big' open subset
  and global  Gromov-Hausdorff convergence. In Section 5, we prove
  Theorem \ref{t1.1},
Theorem \ref{t1.01+} and Corollary \ref{c1.2}.  In Appendix A, we
supply basic properties on Gromov-Hausdorff convergence used in
Section 4. In Appendix B (written by Mark Gross), some bounds for
volumes of Calabi-Yau manifolds along a smoothing are provided which
are used in the proof of Theorem \ref{t4.2+}.

\vskip 7mm

\noindent {\bf Acknowledgement:}   The authors  are  grateful to
 Professor Mark Gross for writing Appendix B and very   helpful discussions and suggestions.  Some results of this
 paper were  obtained during the second author's visiting of  University of California, San  Diego. The second
 author  would like to  thank Professor  M. Gross and mathematics department of UCSD
 for the warm
 hospitality.  The second
 author  would like to  thank  Professor Wei-Dong Ruan for
 discussions with him  during the second
 author's  serving  a postdoc position  in  Korea Advanced Institute of Science and
 Technology which  contribute  some arguments used in Section 5.

 \vskip 10mm

  \section{A Priori  Estimate}
\vskip2mm

In this section, we obtain an estimate  for diameters of Ricci-flat
Calabi-Yau manifolds along a  smoothing, which plays a key role in
our  $C^{0}$-estimate in  the proof of Theorem \ref{t4.2+}.

  \vskip2mm
    \begin{theorem}\label{t1.1+}  Let  $M_{0}$ be a projective  $n$-dimensional
variety  with singular
set $S$. Assume that $M_{0}$  admits  a smoothing $\pi:
\mathcal{M}\rightarrow \Delta$ over the unit disc $\Delta\subset
\mathbb{C}$ such that   $\mathcal{M}$ admits an ample line bundle
$\mathcal{L}$, and the relative canonical bundle is trivial, i.e.,
$\mathcal{K}_{\mathcal{M}/\Delta}\cong \mathcal{O}_{\mathcal{M}}$.
 Let   $ \Omega_{t}$ be  a relative holomorphic volume form, i.e.,
   a   nowhere vanishing  section of  $ \mathcal{K}_{\mathcal{M}/\Delta}$,
   and let   $\tilde{g}_{t}$ be   the unique Ricci-flat K\"{a}hler  metric with
K\"{a}hler form $\tilde{\omega}_{t}\in
c_{1}(\mathcal{L})|_{M_{t}}\in H^{1,1}(M_{t}, \mathbb{R})$, for
$t\in \Delta\backslash \{0\}$. Then the  diameter of $(M_{t},
\tilde{g}_{t}) $ satisfies  that
      $${\rm diam}_{\tilde{g}_{t}}(M_{t})\leq 2+ D(-1)^{\frac{n^{2}}{2}} \int_{M_{t}} \Omega_{t}\wedge  \bar{\Omega}_{t} ,
    $$ where   $D$ is  a constant  independent of $t$.
    \end{theorem}

\vskip2mm
 \begin{proof}
 Recall   that   $\mathcal{M}$
is an   $(n+1)$-dimensional variety  with  an embedding
$\mathcal{M}\hookrightarrow \mathbb{CP}^{N}\times \Delta$ such that
$\mathcal{L}^{m}=\mathcal{O}_{\Delta}(1)|_{\mathcal{M}}$ for an
$m\geq 1$,
 $\pi$ is
   the restriction to $\mathcal{M}$   of
 the projection from $\mathbb{CP}^{N}\times \Delta $ to $\Delta$,
 which is
a proper surjection such that    the rank of $\pi_{*}$ is 1 on $
\mathcal{M}\backslash S$. Then    $M_{t}=\pi^{-1}(t)$ is a smooth
Calabi-Yau manifold for any $t\in \Delta\backslash \{0\}$.
 Denote
$$\omega_{t}=\frac{1}{m}\omega_{FS}|_{M_{t}},$$
 where $\omega_{FS}$ is the standard
Fubini-Study metric on $\mathbb{CP}^{N}$, and  $g_{t}$ is  the
corresponding K\"{a}hler  metric of $ \omega_{t} $. Note that $
\tilde{\omega}_{t}$ satisfies the Monge-Amp\`{e}re equation
\begin{equation}\label{e001.01+} \tilde{\omega}_{t}^{n}=(-1)^{\frac{n^{2}}{2}}e^{\sigma_{t}}\Omega_{t}\wedge
\overline{\Omega}_{ t}, \ \ {\rm where} \ \
e^{\sigma_{t}}=V\left((-1)^{\frac{n^{2}}{2}}\int_{M_{t}}\Omega_{t}\wedge
\overline{\Omega}_{ t}\right)^{-1},
\end{equation} where $V=n!{\rm Vol}_{g_{t}}(M_{t})$  is a
constant independent of $t$.

  For  $p\in M_{0}\backslash S$, there are coordinates $z_{0}, \cdots,  z_{n}$ on  a
 neighborhood $U$ of $p$ in $ \mathcal{M}$ such that $t=\pi(z_{0}, \cdots,
 z_{n})=z_{0}$ and $p=(0, \cdots, 0)$. There is a $r_{0}>0$ such that $ \Delta^{1}\times
 \Delta^{n}\subset U$,  where $\Delta^{1}=\{|t|< r_{0}\}\subset
 \Delta$,  $\Delta^{n}=\{|z_{j}|< r_{0},  j =1, \cdots, n\}\subset \mathbb{C}^{n}$, and $\{t\}\times \Delta^{n} \subset
 M_{t} $.
   Note that locally
   $\omega_{t}$ and $\tilde{\omega}_{t}$ are families of K\"{a}hler
   forms on $\Delta^{n}\subset \mathbb{C}^{n}$, and there is a constant $C_{1}$ independent of
   $t$ such that
   \begin{equation}\label{e001.1+} C_{1}^{-1}\omega_{E} \leq\omega_{t}\leq
   C_{1} \omega_{E}, \ \
   \end{equation}  where $ \omega_{E}=\sqrt{-1}\partial\bar{\partial}\sum\limits_{i=1}^{n}
    |z_{i}|^{2}$ is the standard Euclidean K\"{a}hler form on $\Delta^{n} $, and $g_{E}$ denotes the corresponding
   Euclidean K\"{a}hler  metric.  We need the following fact, which
      is a simplified  version of   Lemma 1.3 in
\cite{DPS}.    For completeness, we shall  sketch a proof.

\vskip2mm

 \begin{lemma}[Lemma 1.3 in
\cite{DPS}]\label{cl1.2+}  For any  $ \delta >0$, and any $t\in
\Delta^{1}\backslash\{0\}$, there is an
 open subset  $U_{t,\delta}$ of $\Delta^{n}$
  such that $$
{\rm Vol}_{g_{t}}(U_{t,\delta})\geq {\rm
Vol}_{g_{t}}(\Delta^{n})-\delta, \ {\rm  \ } \ {\rm
diam}_{\tilde{g}_{t}}(U_{t,\delta}) \leq
\hat{C}\delta^{-\frac{1}{2}},$$ where $\hat{C}$ is  a constant
independent of $t$.
  \end{lemma}
\vskip2mm

\begin{proof}
Let  $dv_{E}=(-1)^{ \frac{n}{2}}dz^{1}\wedge d
\overline{z}^{1}\wedge \cdots \wedge dz^{n}\wedge d
\overline{z}^{n}$ be  the standard Euclidean volume form on
$\Delta^{n}$ and for any  $x_{1}, x_{2}\in \Delta^{n}$, let
$[x_{1},x_{2}]\subset \Delta^{n}$ be  the segment connecting $x_{1}$
and $ x_{2}$.
 By Fubini's Theorem, the Cauchy-Schwarz inequality and (\ref{e001.1+}),  we
 have  \begin{eqnarray*}& & \ \int_{\Delta^{n}\times
\Delta^{n}}{\rm length}_{\tilde{g}_{t}}([x_{1},x_{2}])^{2}dv_{E}
(x_{1})dv_{E} (x_{2})  \\ & \leq &  \|x_{2}-x_{1}
\|_{E}^{2}\int^{1}_{0}ds \int_{\Delta^{n}\times \Delta^{n}} {\rm tr
}_{\omega_{E}}\tilde{\omega}_{t}((1-s)x_{1}+sx_{2})dv_{E}
(x_{1})dv_{E} (x_{2})
\\ & \leq & 2^{2n}{\rm diam}^{2}_{g_{E}}(\Delta^{n}) {\rm Vol}_{g_{E}}(\Delta^{n})
\int_{\Delta^{n}}\tilde{\omega}_{t}\wedge \omega_{E}^{n-1}
\\ & \leq & C_{2}  \int_{\Delta^{n}}\tilde{\omega}_{t}\wedge \omega_{t}^{n-1}\\
& \leq & C_{2}\int_{M_{t}}\tilde{\omega}_{t}\wedge \omega_{t}^{n-1}
   =\bar{C},
\end{eqnarray*} where $\bar{C}$ is a constant  independent of $t$. The
second inequality is obtained by   integrating first with respect to
$y=(1-s)x_{1}$ when $s\leq \frac{1}{2}$, then  with
 respect to  $y=sx_{2}$ when $s\geq \frac{1}{2}$,   since  $dv_{E}
(x_{i})\leq 2^{2n} dv_{E} (y)$.
 If
$$S_{t}=\{ (x_{1},x_{2}) \in \Delta^{n}\times \Delta^{n} | {\rm
length}^{2}_{\tilde{g}_{t}}([x_{1},x_{2}])> \bar{C} \delta^{-1}
\},$$ then ${\rm Vol}_{g_{E}\times g_{E}}(S_{t})<  \delta.$ Let
$S_{t}(x_{1})=\{x_{2}\in \Delta^{n} |\ \  (x_{1},x_{2}) \in S_{t}
\}, $  and let   $ Q_{t}=\{ x_{1}\in \Delta^{n} |\ \  {\rm
Vol}_{g_{E}}(S_{t}(x_{1}))\geq \frac{1}{2}{\rm
Vol}_{g_{E}}(\Delta^{n})\}. $ By Fubini's Theorem,  $$ {\rm
Vol}_{g_{E}}(Q_{t})< 2\delta {\rm Vol}_{g_{E}}^{-1}(\Delta^{n}) , \
\ {\rm } \ \ {\rm Vol}_{g_{E}}(S_{t}(x_{j}))< \frac{1}{2}{\rm
Vol}_{g_{E}}(\Delta^{n}),$$ for any $x_{1},x_{2}\in
\Delta^{n}\backslash Q_{t}$. Thus $(\Delta^{n}\backslash
S_{t}(x_{1}) )\bigcap (\Delta^{n}\backslash S_{t}(x_{2}))$ is not
empty. If $y\in (\Delta^{n}\backslash S_{t}(x_{1})) \cap  (
\Delta^{n}\backslash S_{t}(x_{2}))$, then $ (x_{1},y), (x_{2},y)\in
(\Delta^{n}\times \Delta^{n}) \backslash S_{t}$, and
$$ {\rm length}^{2}_{\tilde{g}_{t}}([x_{1},y]\cup [y,x_{2}])\leq 2
\bar{C} \delta^{-1} ,$$ and therefore
$${\rm diam}_{\tilde{g}_{t}}^{2}(\Delta^{n}\backslash Q_{t})\leq 2
\bar{C} \delta^{-1} .  $$  If we denote
$U_{t,\delta}=\Delta^{n}\backslash Q_{t}$, then by (\ref{e001.1+})
we derive
$${\rm Vol}_{g_{t}}( \Delta^{n} \backslash U_{t,\delta})= {\rm
Vol}_{g_{t}}(Q_{t}) \leq C_{3}{\rm Vol}_{g_{E}}(Q_{t})<2C_{3} {\rm
Vol}_{g_{E}}^{-1}(\Delta^{n})\delta,$$ where $ C_{3}>0$ is   a
constant independent of $t$.   By replacing $\delta$ with  $(2C_{3}
)^{-1} {\rm Vol}_{g_{E}}(\Delta^{n})\delta$, we obtain the desired
conclusion.
\end{proof}

\vskip2mm

We return to the proof of Theorem  \ref{t1.1+}.   Let
$\delta_{t}=\frac{1}{2}{\rm Vol}_{g_{t}}(\Delta^{n})$, and let
$p_{t}\in U_{t,\delta_{t}}$. By
 (\ref{e001.1+}), we get
$$\delta_{t}\geq \frac{C_{4}}{2}{\rm Vol}_{g_{E}}(\Delta^{n})=\bar{\delta}$$
and thus  $U_{t,\delta_{t}}\subset B_{\tilde{g}_{t}}(p_{t},r)
 $, where  $
r=\max \{ 1, 2\hat{C} \bar{\delta}^{- \frac{1}{2}}\}$  and $\hat{C}$
is the constant in Lemma \ref{cl1.2+}. Since $U\subset
\mathcal{M}\backslash S$, there is a constant $\kappa_{U}>0$ such
that
$$(-1)^{\frac{n^{2}}{2}}\Omega_{t}\wedge \overline{\Omega}_{ t}\geq
\kappa_{U}\omega_{t}^{n}$$ on $U\cap M_{t}$.  By (\ref{e001.01+}),
we derive
 \begin{eqnarray*} {\rm Vol}_{\tilde{g}_{t}}(B_{\tilde{g}_{t}}(p_{t},r)) \geq
 {\rm Vol}_{\tilde{g}_{t}}(U_{t,\delta_{t}})  &
 = &
\frac{(-1)^{\frac{n^{2}}{2}}}{n!}e^{\sigma_{t}}\int_{U_{t,\delta_{t}}}\Omega_{t}\wedge
\overline{\Omega}_{ t}\\ & \geq &
\frac{\kappa_{U}e^{\sigma_{t}}}{n!}\int_{U_{t,\delta_{t}}} \omega_{t}^{n}\\
& = & \kappa_{U}e^{\sigma_{t}}{\rm Vol}_{g_{t}}(U_{t,\delta_{t}})\\
&\geq & \frac{\kappa_{U}e^{\sigma_{t}}}{2}{\rm
Vol}_{g_{t}}(\Delta^{n})\\ &\geq & C_{5}e^{\sigma_{t}}{\rm
Vol}_{g_{E}}(\Delta^{n})= C_{6}e^{\sigma_{t}},\end{eqnarray*} where
$C_{6}$ is  a constant  independent of $t$. By   Bishop-Gromov
relative volume  comparison, we obtain
$${\rm Vol}_{\tilde{g}_{t}}(B_{\tilde{g}_{t}}(p_{t},1))  \geq
\frac{1}{r^{2n}}{\rm
Vol}_{\tilde{g}_{t}}(B_{\tilde{g}_{t}}(p_{t},r))\geq
\frac{C_{6}}{r^{2n}}e^{\sigma_{t}}.
$$  In the rest of the proof, we need the following lemma.

\vskip2mm

\begin{lemma} [ Theorem 4.1 of Chapter 1 in \cite{SY2} and Lemma 2.3 in \cite{Pa}]\label{t3.03}
 Let $(M,g)$ be  a $2n$-dimensional  compact
Riemannian manifold with nonnegative  Ricci curvature. Then  for any
  $p\in M$ and any   $1< R< {\rm diam}_{g}(M)$, we have
$$ \frac{{\rm Vol}_{g}(B_{g}(p, 2R+2))}{{\rm Vol}_{g}(B_{g}(p, 1))}\geq
\frac{R-1}{2n}.$$\end{lemma}

\vskip2mm

 By letting $R=\frac{1}{2}{\rm diam}_{\tilde{g}_{t}}(M_{t})$, we
obtain $$ {\rm diam}_{\tilde{g}_{t}}(M_{t})\leq 2+ 8n \frac{{\rm
Vol}_{\tilde{g}_{t}}(M_{t})}{{\rm
Vol}_{\tilde{g}_{t}}(B_{\tilde{g}_{t}}(p_{t},1))}\leq 2+
De^{-\sigma_{t}},$$ where $D$ is  a constant  independent of $t$. We
 conclude the proof   by (\ref{e001.01+}).
  \end{proof}

   \vskip2mm
   The following  is a consequence of Theorem  \ref{t1.1+}  and  Theorem  \ref{B.1}.
  \vskip2mm
    \begin{corollary}\label{c1.1+}
    Let   $M_{0}$, $\mathcal{M}$,  $\mathcal{L}$,  $ \Omega_{t}$, and   $\tilde{g}_{t}$
     be  as in Theorem \ref{t1.1+}. If in addition  we assume  that $M_{0}$ is
     a Calabi-Yau $n$-variety,   then the  diameter of $(M_{t},
\tilde{g}_{t}) $ has  a uniform bound
      $${\rm diam}_{\tilde{g}_{t}}(M_{t})\leq  D,  $$ where $D$ is  a constant  independent of $t$.
    \end{corollary}

\vskip4mm

 \section{Proof of Theorem \ref{t4.2+}}
 \vskip2mm

 Let  $M_{0}$ be an $n$-dimensional Calabi-Yau variety  with singular
set $S$.   Assume that $M_{0}$  admits  a smoothing $\pi:
\mathcal{M}\rightarrow \Delta$ over the unit disc $\Delta\subset
\mathbb{C}$ such that  $\mathcal{M}$ admits an ample line bundle
$\mathcal{L}$, and the relative canonical bundle is trivial, i.e.,
$\mathcal{K}_{\mathcal{M}/\Delta}\cong \mathcal{O}_{\mathcal{M}}$.
Following the discussion at the beginning of the proof of Theorem
\ref{t1.1+},     let
$$\omega_{t}=\omega_{h}|_{M_{t}}=\frac{1}{m}\omega_{FS}|_{M_{t}},
\ \ {\rm and} \ \ \omega_{h}=\sqrt{-1}\partial\bar{\partial}
|t|^{2}+\frac{1}{m}\omega_{FS},$$ for any $t\in\Delta$, where
$\omega_{FS}$ is the standard Fubini-Study metric on
$\mathbb{CP}^{N}$, and   $g_{t}$ is   the corresponding K\"{a}hler
metric of $ \omega_{t} $.  Let $ \Omega_{t}$ be a relative
holomorphic volume form, i.e.,     a  nowhere vanishing section of $
\mathcal{K}_{\mathcal{M}/\Delta}$.  Yau's proof of Calabi's
conjecture (\cite{Ya1})  asserts that    there is a unique
Ricci-flat K\"{a}hler  metric  $\tilde{g}_{t}$  with K\"{a}hler form
$\tilde{\omega}_{t}\in [\omega_{t}]= c_{1}(\mathcal{L})|_{M_{t}}\in
H^{1,1}(M_{t}, \mathbb{R})$ for $t\in \Delta\backslash \{0\}$, i.e.,
there is a   unique function $\varphi_{t}$ on $M_{t}$ satisfying
that $\tilde{\omega}_{t}=\omega_{t}+
\sqrt{-1}\partial\overline{\partial}\varphi_{t}$, and
\begin{equation}\label{e51.3}(\omega_{t}+
\sqrt{-1}\partial\overline{\partial}\varphi_{t})^{n}=(-1)^{\frac{n^{2}}{2}}e^{\sigma_{t}}\Omega_{t}\wedge
\overline{\Omega}_{ t},
 \ \ {\rm with} \ \ \sup_{M_{t}}\varphi_{t} =0
\end{equation} where   $$\sigma_{t}=\log\left(n!V((-1)^{\frac{n^{2}}{2}}\int_{M_{t}}\Omega_{t}\wedge
\overline{\Omega}_{ t})^{-1}\right) $$ and  $V={\rm
Vol}_{\tilde{g}_{t}}(M_{t})$.

By  Theorem  \ref{B.1}, on $ M_{t}$  we have
\begin{equation}\label{l003.3+}
(-1)^{\frac{n^{2}}{2}}\Omega_{t}\wedge \overline{\Omega}_{ t}\geq
\kappa \omega_{t}^{n},
\end{equation}   and
\begin{equation}\label{e1.1+}\int_{M_{t}}(-1)^{\frac{n^{2}}{2}}\Omega_{t}\wedge
\overline{\Omega}_{ t}\leq \Lambda, \end{equation} where $\kappa
>0$ and $\Lambda >0$ are  constants   independent of $t\in \Delta\backslash\{0\}$.
  Thus   there is a constant $C_{1}>0$ independent of $t$ such
that \begin{equation}\label{e51.3+}0<C_{1}^{-1}\leq e^{\sigma_{t}}
\leq C_{1}.
\end{equation}

Note that $(M_{t}, \tilde{g}_{t})$ satisfies  that $${\rm
Ric}_{\tilde{g}_{t}}\equiv 0, \ \ {\rm
Vol}_{\tilde{g}_{t}}(M_{t})\equiv V, \ \ {\rm and} \ \  {\rm
diam}_{\tilde{g}_{t}}(M_{t})\leq D,$$ where the upper bound of
diameters  is from   Corollary \ref{c1.1+}. By \cite{Cr}, \cite{Ga}
and \cite{Li}, $(M_{t}, \tilde{g}_{t})$ has uniform Sobolev
constants, i.e.,  constants $\bar{C}_{S,1}>0$ and $\bar{C}_{S,2}
>0$ independent of $t$ such that for any $t\neq 0$ and
 any  smooth function $\chi$ on $M_{t}$,  \begin{equation}\label{e001.2+} \|\chi\|_{L^{\frac{4n}{2n-2}}
 (\tilde{g}_{t})}^{2}\leq \bar{C}_{S,1}
 (\|d\chi\|_{L^{2}(\tilde{g}_{t})}^{2}+\|\chi\|_{L^{2}
 (\tilde{g}_{t})}^{2})
, \end{equation}   and   if
  $\int_{M_{t}}\chi dv_{\tilde{g}_{t}}=0 $,   \begin{equation}\label{e001.3+}
\|\chi\|_{L^{\frac{4n}{2n-2}}
 (\tilde{g}_{t})}^{2}\leq \bar{C}_{S,2}
 \|d\chi\|_{L^{2}(\tilde{g}_{t})}^{2}.  \end{equation}

 Now  following  the standard  Moser iteration argument in    \cite{Ya1} with
  a trick  inspired by  \cite{To}, we are able to get a uniform
 $C^{0}$-estimate of the potential function $\varphi_{t}$.

\vskip2mm
\begin{lemma}\label{l5.3+}
There is a constant $C>0$ independent of $t\in \Delta\backslash
\{0\}$ such that
$$\|\varphi_{t}\|_{C^{0}(M_{t})}\leq  C.
$$
\end{lemma}
\vskip2mm

\begin{proof}
Let
$$f_{t}=\log\left(\frac{(-1)^{\frac{n^{2}}{2}}e^{\sigma_{t}}\Omega_{t}\wedge
\overline{\Omega}_{ t}}{\omega_{t}^{n}}\right), \ \ \ \
\tilde{\varphi}_{t}=\int_{M_{t}}\varphi_{t}dv_{\tilde{g}_{t}}-
\varphi_{t}.$$  Then (\ref{e51.3}) shows that
$$\omega_{t}^{n}=e^{-f_{t}}\tilde{\omega}_{t}^{n}=(\tilde{\omega}_{t}+
\sqrt{-1}\partial\overline{\partial}\tilde{\varphi}_{t})^{n},
 \ \ {\rm with} \ \ \int_{M_{t}}\tilde{\varphi}_{t} dv_{\tilde{g}_{t}}=0.
 $$  By (\ref{l003.3+})   and (\ref{e51.3+}), there is a constant $C_{2}>0$ independent of $t$
   such that $$ e^{-f_{t}}=\left(\frac{(-1)^{\frac{n^{2}}{2}}e^{\sigma_{t}}\Omega_{t}\wedge
\overline{\Omega}_{ t}}{\omega_{t}^{n}}\right)^{-1}\leq C_{2}. $$
 Now  we follow the standard  Moser iteration argument in    \cite{Ya1} (cf. \cite{Au}).

 A direct calculation shows that
 \begin{equation}\label{ea1.3+}\int_{M_{t}}|d|\tilde{\varphi}_{t}|^{\frac{p}{2}}|^{2}dv_{\tilde{g}_{t}} \leq  \frac{np^{2}}{4(p-1)}
   \int_{M_{t}}|1-e^{-f_{t}}||\tilde{\varphi}_{t}|^{p-1}dv_{\tilde{g}_{t}}
   \leq Ap \int_{M_{t}}|\tilde{\varphi}_{t}|^{p-1}dv_{\tilde{g}_{t}},\end{equation}
   for any $p\geq 2$ (cf.   (15) in Chapter
   7 of  \cite{Au}), where $A>0$ is a constant independent of $t$.  For $p=2$,   by
   (\ref{e001.3+}), (\ref{ea1.3+})
   and H\"{o}lder's inequality we see  that \begin{eqnarray*}\|\tilde{\varphi}_{t}\|_{L^{\frac{4n}{2n-2}}
 (\tilde{g}_{t})}^{2}& \leq &  \bar{C}_{S,2}
 \|d\tilde{\varphi}_{t}\|_{L^{2}(\tilde{g}_{t})}^{2} \\ &\leq &  2A\bar{C}_{S,2}
 \int_{M_{t}}|\tilde{\varphi}_{t}|dv_{\tilde{g}_{t}} \\ & \leq & 2A\bar{C}_{S,2}V^{\frac{2n+2}{4n}}
 \|\tilde{\varphi}_{t}\|_{L^{\frac{4n}{2n-2}}
 (\tilde{g}_{t})},  \end{eqnarray*} and thus $$\|\tilde{\varphi}_{t}\|_{L^{\frac{4n}{2n-2}}
 (\tilde{g}_{t})}\leq \hat{C},$$ where  $\hat{C}$ is   a constant independent of $t$.
 For  $p>2$, by (\ref{e001.2+}),  (\ref{ea1.3+})
   and H\"{o}lder's inequality  we see that   \begin{eqnarray*}\|\tilde{\varphi}_{t}\|_{L^{\frac{2np}{2n-2}}
 (\tilde{g}_{t})}^{p}& = &\||\tilde{\varphi}_{t}|^{\frac{p}{2}}\|_{L^{\frac{4n}{2n-2}}
 (\tilde{g}_{t})}^{2} \\ & \leq &  \bar{C}_{S,1}
 (\|d|\tilde{\varphi}_{t}|^{\frac{p}{2}}\|_{L^{2}(\tilde{g}_{t})}^{2}+\||\tilde{\varphi}_{t}|^{\frac{p}{2}}\|_{L^{2}
 (\tilde{g}_{t})}^{2})\\ & \leq & \bar{C}_{S,1}(pAV^{\frac{1}{p}}+\|\tilde{\varphi}_{t}\|_{L^{p}
 (\tilde{g}_{t})})\|\tilde{\varphi}_{t}\|_{L^{p}
 (\tilde{g}_{t})}^{p-1}. \end{eqnarray*} Let
 $p_{0}=\frac{4n}{2n-2}$,
  $p_{k+1}=\frac{2n}{2n-2}p_{k}$ ($k\geq 0$), let  $\hat{C}_{0}=\hat{C}$ and let
   $\hat{C}_{k+1}=\bar{C}_{S,1}^{\frac{1}{p_{k}}}
  (p_{k}AV^{\frac{1}{p_{k}}}+1)^{\frac{1}{p_{k}}}\hat{C}_{k}$ if
  $\hat{C}_{k}>1$. Otherwise, let
  $\hat{C}_{k+1}=\bar{C}_{S,1}^{\frac{1}{p_{k}}}
  (p_{k}AV^{\frac{1}{p_{k}}}+1)^{\frac{1}{p_{k}}}$.  Then  $\|\tilde{\varphi}_{t}\|_{L^{p_{k}}
 (\tilde{g}_{t})}\leq \hat{C}_{k}<C_{3}$,  a constant $C_{3}>0$ independent of
 $k$ and $t$.  By letting $k\rightarrow\infty$, we have $$\|\tilde{\varphi}_{t}\|_{C^{0}
 (M_{t})}\leq C_{3}.$$ Since  there is a  $p_{t}\in M_{t}$ such that $
 \varphi_{t}(p_{t})=0$,  we have
 $$\left|\int_{M_{t}}\varphi_{t}dv_{\tilde{g}_{t}}\right|\leq C_{3},  \ \ {\rm and
 } \ \ \|\varphi_{t}\|_{C^{0}
 (M_{t})}\leq C,$$ where $C>0$ is   a constant  independent
 of $t$.
\end{proof}

\vskip2mm

The $C^{2}$-estimate  for $\varphi_{t}$ is obtained by the
 same  arguments as in proof of  Lemma 5.2  in  \cite{RZ}. For the
completeness, we   present it here.

\vskip2mm

\begin{lemma}\label{l5.4}
For any compact subset $K\subset \mathcal{M}\backslash S$, there
exists a constant $C_{K}>0$ independent of $t$ such that on $K\cap
M_{t} $,
$$C\omega_{t}\leq
\tilde{\omega}_{t}\leq C_{K}\omega_{t},$$  where $C>0$ is a constant
independent of $t$ and $K$.
\end{lemma}

\vskip2mm
\begin{proof} Let
 $\psi_{t}: (M_{t}, \tilde{\omega}_{t}) \longrightarrow (\mathbb{CP}^{N},
\frac{1}{m}\omega_{FS})$ be the inclusion   map induced by
$\mathcal{M}\subset \mathbb{CP}^{N}\times\Delta $.  The Chern-Lu
inequality says
$$\Delta_{\tilde{\omega}_{t}}\log |\partial \psi_{t}|^{2}\geq \frac{{\rm Ric}_{\tilde{\omega}_{t}}(\partial \psi_{t},
\overline{ \partial \psi_{t}})}{|\partial \psi_{t}|^{2}}-\frac{{\rm
Sec}(\partial \psi_{t}, \overline{
\partial \psi_{t}},
\partial \psi_{t},
\overline{ \partial \psi_{t}})}{|\partial \psi_{t}|^{2}},
$$  where ${\rm
Sec}$ denotes  the holomorphic bi-sectional curvature of
$\frac{1}{m}\omega_{FS}$ (cf.  \cite{Y3}). Note that
$\psi_{t}^{*}\omega_{FS}=\omega_{t} $,  $|\partial
\psi_{t}|^{2}={\rm
tr}_{\tilde{\omega}_{t}}\psi_{t}^{*}\omega_{FS}={\rm
tr}_{\tilde{\omega}_{t}}\omega_{t}=n-\Delta_{\tilde{\omega}_{t}}\varphi_{t}$
and ${\rm Ric}_{\tilde{\omega}_{t}}=0$. Thus we have that
 $$\Delta_{\tilde{\omega}_{t}}(\log {\rm tr}_{\tilde{\omega}_{t}}\omega_{t}-2\overline{R}\varphi_{t})
\geq -2\overline{R}n+\overline{R}{\rm
tr}_{\tilde{\omega}_{t}}\omega_{t}.
$$where $\overline{R}$ is  a constant  depending only the upper bound of ${\rm
Sec} $. By the  maximum principle  and
 Lemma \ref{l5.3+}, there is an $x\in M_{t}$ such that
${\rm tr}_{\tilde{\omega}_{t}}\omega_{t}(x)\leq 2n$,
 $${\rm tr}_{\tilde{\omega}_{t}}\omega_{t}\leq 2ne^{2\overline{R}(\varphi_{t}-\varphi_{t}(x))}\leq C \ \ \
{\rm and} \ \ \  \omega_{t}\leq C \tilde{\omega}_{t},  $$
 where $C>0$ is  a constant  independent of $t$. Note that for any compact subset $K\subset \mathcal{M}\backslash
S$, by (\ref{e51.3+}) and the  compactness of $K$ there exists a
constant $C'_{K}>0$ independent of $t$  such that on $K\cap M_{t} $
$$
\tilde{\omega}_{t}^{n}=e^{\sigma_{t}}(-1)^{\frac{n^{2}}{2}}\Omega_{t}\wedge
\overline{\Omega}_{ t} \leq C'_{K}\omega_{t}^{n}.$$ Then we obtain
that
$$C\omega_{t}\leq \tilde{\omega}_{t}\leq C_{K}\omega_{t}.$$
\end{proof}
\vskip2mm

Now we are ready to  prove  Theorem \ref{t4.2+}.

 \vskip 2mm

\begin{proof}[Proof of Theorem \ref{t4.2+}]
   In \cite{EGZ}, it is proved that there is a unique
 bounded  function $\hat{\varphi}_{0}$ on $M_{0}$ such that $\hat{\varphi}_{0}$  is smooth
on $M_{0}\backslash S$ and satisfies
\begin{equation}\label{6.22}
 (\omega_{0}+
\sqrt{-1}\partial \overline{\partial}\hat{\varphi}_{0}
)^{n}=(-1)^{\frac{n^{2}}{2}}e^{\hat{\sigma}_{0}}\Omega_{0}\wedge
\overline{\Omega}_{ 0} , \ \ \ \sup_{M_{0}}\hat{\varphi}_{0}=0,
\end{equation}
 in the distribution sense, where $\hat{\sigma}_{0}
$ is a constant. Note that  $\omega=\omega_{0}+\sqrt{-1}\partial
\overline{\partial}\hat{\varphi}_{0}$ is the unique  singular
Ricci-flat
 K\"ahler   form with $\|\hat{\varphi}_{0}\|_{L^{\infty}}\leq C $.
 Let $F:(M_{0}\backslash S)\times \Delta
 \longrightarrow  \mathcal{M}$ be a smooth embedding  such that  $F((M_{0}\backslash S)\times
 \{t\})\subset M_{t}$ and $F|_{(M_{0}\backslash S)\times
 \{0\}}: M_{0}\backslash S \longrightarrow M_{0}\backslash S$ is the identity map.
   Let $K_{1}\subset \cdots \subset K_{i} \subset \cdots \subset
M_{0}\backslash S$ be  a sequence of compact  subsets  such that
$M_{0}\backslash S =\bigcup\limits_{i}K_{i}$. On a  fixed  $K_{i}$,
the embedding map
$$F_{K_{i},t}=F|_{K_{i}\times
 \{t\}}: K_{i} \longrightarrow M_{t}$$ satisfies   that $F_{K_{i},t}^{*}\omega_{t}$ $C^{\infty}$-converges to
$\omega_{0}$, and $dF_{K_{i},t}^{-1}J_{t}dF_{K_{i},t}$
$C^{\infty}$-converges to $J_{0}$, where $J_{t}$ (resp.  $J_{0}$) is
the complex structure on $M_{t}$ (resp.  $M_{0}$).

 For a fixed
$K_{i}$,  let $K$ be a compact subset of $\mathcal{M}\backslash S$
such that $F_{K_{i},t_{k}}( K_{i})\subset K$ for $|t_{k}|\ll 1$.
  By (\ref{e51.3+}),   Lemma \ref{l5.3+}  and Lemma \ref{l5.4},
  there exist  constants $C>0$ and  $C_{K}>0$
independent of $t$ such that $C^{-1}\leq \sigma_{t} \leq C$,
$\|\varphi_{t}\|_{C^{0}(M_{t})}\leq C$, and $C^{-1}\omega_{t}\leq
\omega_{t}+\sqrt{-1}\partial \overline{\partial}\varphi_{t} \leq
C_{K}\omega_{t} $ on $K$.  By Theorem 17.14 in \cite{GT}, we have
that  $\|\varphi_{t}\|_{C^{2, \alpha}(M_{t}\cap K)}\leq C_{K}''$ for
a constant $C_{K}''>0$. Furthermore,  by the standard bootstrapping
argument we have that  for any $l>0$,   $\|\varphi_{t}\|_{C^{l,
\alpha}(M_{t}\cap K)}\leq C_{K,l}$ for constants $C_{K,l}>0$
independent of $t$.   By the standard diagonal   arguments and
passing to a subsequence, we see that
$F_{K_{i_{k}},t_k}^{*}\varphi_{t_{k}}$ $C^{\infty}$-converges to a
smooth function  $\varphi_{0}$ on $M_{0}\backslash  S$ with
 $\|\varphi_{0}\|_{L^{\infty}}<C$ and that  $\sigma_{t_{k}} $ converges to a $\sigma_{0}$,
 which satisfies   $$(\omega_{0}+ \sqrt{-1}\partial
\overline{\partial}\varphi_{0}
)^{n}=(-1)^{\frac{n^{2}}{2}}e^{\sigma_{0}}\Omega_{0}\wedge
\overline{\Omega}_{ 0}.$$   Hence $\tilde{\omega}_{0}=\omega_{0}+
\sqrt{-1}\partial \overline{\partial}\varphi_{0}$ is a Ricci-flat
K\"{a}hler  form on
 $M_{0}\backslash S$ with $\|\varphi_{0}\|_{L^{\infty}}<C$.
  By the uniqueness of the solution of
(\ref{6.22}),  $ \varphi_{0}= \hat{\varphi}_{0}$ and
$\sigma_{0}=\hat{\sigma}_{0}$. The uniqueness of $\omega$ and the
standard compactness argument
 imply    that $F|_{M_{0}\backslash S\times
 \{t\}}^{*}\tilde{\omega}_{t}$  (resp. $F|_{M_{0}\backslash S\times
 \{t\}}^{*}\tilde{g}_{t}$) $C^{\infty}$-converges
to $\omega$ (resp. $g$)  when $t\rightarrow 0$.

The diameter estimate is obtained  by Corollary \ref{c1.1+}.
\end{proof}

\vskip4mm

\section{An  Almost Gauge Fixing Theorem}

\vskip2mm

Let $M$ be a compact $n$-manifold, and let $g_k$ be a sequence of
Riemannian metrics on $M$. Assume that the Ricci curvature, volume
and diameter of $g_k$ satisfy

\noindent i) $|\text{Ric}(g_k)|\le 1, \text{Vol}_{g_k}(M)\ge V>0$
and $\text{diam}_{g_k}(M)\le D$.

By the Gromov's pre-compactness theorem, we may assume

\noindent ii) $(M,g_k) \stackrel{d_{GH}}\longrightarrow  (X,d_X)$,
where $(X,d_X)$ is a compact metric space.

Suppose, in addition,

\noindent iii) $E$ is a closed  subset  of Hausdorff  dimension $\le
n-2$, and there is a (non-complete) Riemannian metric $g_\infty$ on
$M\backslash E$ such that $g_k$ converges to $g_\infty$ in the
$C^\infty$-sense on any compact subset $K\subset M\backslash E$.

Because $M\backslash E$ is path connected, $g_\infty$ induces the
the Riemannian distance structure  defined by
$$d_{g_\infty}(x,y)=\inf_{\gamma \text{ continuous}}\{{\rm length}_{g_\infty}(\gamma), \,
\,\, \gamma:[0,1]\to M\backslash E,\,\, \gamma(0)=x,\gamma(1)=y\}.$$
Let $\overline{(M\backslash E,g_\infty)}$ denote the metric
completion of $(M\backslash E,d_{g_\infty})$. Let $S_X\subset X$
denote  the subset consisting of points $x\in X$ such that there is a
sequence $x_k\in E\subset(M,g_k)$ and $x_k\to x$ (see comments at
the end of Appendix A). It is clear that $S_X\subset X$ is a closed
subset and thus $S_X$ is compact.

The main effort of this section is to prove the following result.

\vskip2mm

\begin{theorem}\label{t2.01}
Let $M$, $g_{k}$, $g_{\infty}$, $d_{g_\infty}$,   $E$, $(X,d_{X})$
and $S_X $  be as above. Then there is a continuous surjection
$f: \overline{(M\backslash E,d_{g_\infty})}\to (X,d_X)$ such that
$f: (M\backslash E,d_{g_\infty})\to (X\backslash S_X,d_X)$ is a
homeomeomorphism and a local isometry, i.e., for any $x\in
M\backslash E$, there is an open neighborhood of $x$, $U\subset
M\backslash E$, such that $f: (U,d_{g_\infty}|_{U})\to
(f(U),d_X|_{f(U)})$ is an isometry.
\end{theorem}

\vskip2mm

\begin{proof} We first construct a dense subset $A\subseteq X\backslash
S_X$ and define a local isometric embedding $h: (A, d_{X})\to
(M\backslash E,d_{g_\infty})$ such that $f(A)$ is dense. Then we
will show that $f=h^{-1}: h(A)\to X\backslash S_X$ extends uniquely
to a continuous surjection $f: \overline{(M\backslash
E,g_\infty)}\to (X,d_X)$ such that $f$ is a homeomorphism and a
local isometric embedding on $(M\backslash E,d_{g_\infty})$.

Without loss of generality we may assume that for all $k\ge j$,
$d_{GH}((M,g_k),(M,g_j))<2^{-j}$. Let $\phi_j: (M,g_{j+1})\to
(M,g_j)$ denote an $2^{-j}$-Gromov-Hausdorff approximation. Then
$\phi_j^{j+s}=\phi_j\circ \cdots \circ \phi_{j+s-1}: (M,g_{j+s})\to
(M,g_{j+s-1})\to \cdots \to (M,g_j)$ is an
$2^{-j+1}$-Gromov-Hausdorff approximation. Recall  that there is an
admissible metric $d_Z$ on the disjoint union
$Z=(\coprod\limits_{k=1}^\infty(M,g_k))\coprod (X,d_X)$ such that
$(M,g_k)\stackrel{d_{Z,H}}\longrightarrow  (X,d_X)$ (see Appendix A,
Proposition A.1).

Let $\epsilon_j=j^{-1}$, $j=1, 2, \cdots $. For $\epsilon_1$ and
each $g_k$, take a finite $\epsilon_1$-net  $\{x^k_{i_1}\}\subset
(M\backslash E,g_k)$ such that
\begin{equation}\label{2.1.1}|\{x^k_{i_1}\}|\leq c_1', \ \ {\rm  and } \end{equation}
\begin{equation}\label{2.1.2}d_{g_k}(\{x^k_{i_1}\},E)\ge \frac{\epsilon_1}2,
\ \ {\rm where } \ \ d_{g_k}(\{x^k_{i_1}\},E)=\min
\{d_{g_k}(x^k_j,y), \,\, x^k_j\in \{x^k_{i_1}\}, \, y\in
E\}.\end{equation} We may assume, passing to a subsequence if
necessary, that $\{x^k_{i_1}\}\stackrel{d_{H,Z}}\longrightarrow
 \{x_{i_1}\}_{i_1=1}^{c_1}\subset (X,d_X)$, where $c_1=| \{x_{i_1}\}| $.
Then by (\ref{2.1.2}), $\{x_{i_1}\}_{i_1=1}^{c_1}\subset X\backslash
S_X$. We claim that there is $\bar{k}_1>0$ such that for all $k\ge
\bar{k}_1$, $\{\phi_{\bar{k}_1}^k(x^k_{i_1})\}\subset K_1=M
\backslash B_{g_{\bar{k}_1}}(E,\frac{\epsilon_1}4)$, a compact
subset. Here $B_{g_{\bar{k}_1}}(E,\frac{\epsilon_1}4)=\{y\in M|
d_{g_{\bar{k}_1}}(y,E)< \frac{\epsilon_1}4\}$. Assuming the claim,
by iii) we may assume that passing to a subsequence
$\{\phi^k_{\bar{k}_1}(x^k_{i_1})\}\to
\{y_{i_1}\}_{i_1=1}^{c_1}\subset (M\backslash E,g_\infty)$
point-wise, and we denote the corresponding subsequence by
$\{g_{k_1}\}\subset \{g_k\}$.

To verify the claim, we may assume $\bar{k}_1$ large so that for all
$k\ge \bar{k}_1$,
$$d_{Z,H}(\{x^k_{i_1}\},\{x_{i_1}\})<\frac{\epsilon_1}9,\qquad 2^{-\bar{k}_1}
\ll\epsilon_1.$$ For the sake of distinction, let $E_0=E\subset
(M,g_{\bar{k}_1})$. Then
\begin{eqnarray*} d_{g_{\bar{k}_1}}(\{\phi_{\bar{k}_1}^k(x^k_{i_1})\},E_0)&= & d_{Z,H}
(\{\phi_{\bar{k}_1}^k(x^k_{i_1})\},E_0)\\ & \ge &
d_{Z,H}(\{x^{\bar{k}_1}
_{i_1}\},E_0)-d_{Z,H}(\{x^{\bar{k}_1}_{i_1}\},\{\phi_{\bar{k}_1}^k(x^k_{i_1})\})\\&
\ge & \frac{\epsilon_1}2-[d_{Z,H}(\{x^{\bar{k}_1}_{i_1}\},
\{x_{i_1}\})+d_{Z,H}(\{\phi_{\bar{k}_1}^k(x^k_{i_1})\},\{x_{i_1}\})]\\&
\ge &
\frac{\epsilon_1}2-\left[\frac{\epsilon_1}9+d_{Z,H}(\{\phi_{\bar{k}_1}^k(x^k_{i_1})\},
\{x^k_{i_1}\})+d_{Z,H}(\{x^k_{i_1}\},\{x_{i_1}\})\right]\\&\ge &
\frac{\epsilon_1}2-\left(\frac{\epsilon_1}9+\frac{\epsilon_1}9+2^{-\bar{k}_1}\right)
\ge \frac{\epsilon_1}4.\end{eqnarray*}

For $\epsilon_2$ and each $g_{k_1}$, extend $\{x^{k_1}_{i_1}\}$ to
an $\epsilon_2$-dense subset of $(M\backslash E,g_{k_1})$,
$\{x^{k_1}_{i_1}\}\subset \{x^{k_1}_{i_2}\}$, such that for all
$g_{k_1}$  \begin{equation}\label{2.2.1}
d_{g_{k_1}}(x^{k_1}_{i_2},x^{k_1}_{i_2'} )\geq \frac{\epsilon_2}{4},
\ \ \  |\{x^{k_1}_{i_2}\}|\leq c_2', \ \ {\rm and } \end{equation}
\begin{equation}\label{2.2.2}d_{g_{k_1}}(\{x^{k_1}_{i_2}\},E)\ge
\frac{\epsilon_2}2.  \end{equation} Similarly, by (\ref{2.2.1}) and
(\ref{2.2.2}), passing to a subsequence we may assume that
$\{x^{k_1}_{i_2}\}\stackrel{d_{H,Z}}\longrightarrow
\{x_{i_2}\}\subset (X\backslash S_X,d_X)$. Clearly, $\{x_{i_1}\}
\subset \{x_{i_2}\}_{i_2=1}^{c_2}$, where $c_2=|\{x_{i_2}\}| $. By
the argument as in the above, we may assume large $\bar{k}_2>
\bar{k}_1$ such that for all $k\ge \bar{k}_2$,
$\{\phi_{\bar{k}_2}^k(x^k_{i_2})\}\subset K_2=(M\backslash
B_{g_{k_1}}(E,\frac {\epsilon_2}4))$. By the compactness of $K_2$
and iii), we may assume that
$\{\phi_{\bar{k}_2}^k(x^k_{i_2})\}\to
\{y_{i_2}\}_{i_2=1}^{c_1}\subset (M\backslash E,d_{g_\infty})$ point-wise. The
natural identification $\phi_{\bar{k}_1}^k(x^k_{i_1})\leftrightarrow
\phi_{\bar{k}_2}^k(x^k_{i_1})$ induces an injective map,
$\{y_{i_1}\}\hookrightarrow \{y_{i_2}\}$.

Repeating this process and together with a standard diagonal
argument, we obtain a sequence of finite subsets of $(X\backslash
S_X,d_X)$:
$$\{x_{i_1}\}_{i_1=1}^{c_1}\subset \cdots \subset \{x_{i_s}\}_{i_s=1}^{c_s}
\subset \cdots,$$ and a sequence of finite subsets of $(M\backslash
E,g_\infty)$:
\begin{equation}\label{4.2.1++}\{y_{i_1}\}_{i_1=1}^{c_1}\hookrightarrow
\cdots \hookrightarrow \{y_{i_s}\}_{i_s=1}^{c_s} \hookrightarrow
\cdots.\end{equation} Let $A=\bigcup\limits_{s=1}^\infty
\{x_{i_s}\}_{i_s=1}^{c_s}$, and $A_{k_l}=\bigcup\limits_{s=1}
^\infty\{x^{k_l}_{i_s}\}_{i_s=1}^{c_s}$. Since $A_{k_l}$ is dense in
$(M\backslash E,g_{k_l})$ for all $k_l$, $A\subset (X\backslash
S_{X},d_X)$ is a dense subset. Let $Y$ denote the direct limit of
(\ref{4.2.1++}). Then $Y\subseteq M-E$. We now define a map, $f:
A\to (M\backslash E,g_\infty)$ by
$$f(x_{i_s})=[y_{i_s}]=\{y_{i_s}\rightarrow \cdots \rightarrow \cdots \}.$$
It is clear that $f$ is injective since $f$ is injective on each
$\{x_{i_s}\}_{i_s=1}^{c_s}$, and $f(A)$ is dense in $(M\backslash
E,g_\infty)$. From the construction of $f$, we see that $f$ is a
local isometric embedding: for $x\in A$ we may assume that
$x=x_{i_s}$. Since $x_{i_s}\notin S_X$ which is compact subset of
$X$, there is a $r>0$ such that $\bar B_{d_{X}}(x_{i_s},r)\cap
S_X=\emptyset$. Recall that we may assume $\bar{k}_{v}$ large and
$\phi_{\bar{k}_{v}}^{k_l}(x^{k_l}_{i_s}) \subset K$ and
$\phi_{\bar{k}_{v}}^{k_l}(x^{k_l}_{i_s})\to y_{i_s}$ point-wise
with respect to $d_{g_{\infty}}$, where $K\subset M$ is compact such
that $K\cap E=\emptyset$. Clearly, we may assume that $r$ small and
a compact subset $K'\supseteq K$ such that
$B_{g_\infty}([y_{i_s}],r)\subset K'$ and $K'\cap E=\emptyset$. By
iii), $(K',g_{k_l})\to (K',g_\infty)$ in the $C^{\infty}$-sense.
Observe the following two facts:

\noindent (4.5) For $z, z'\in B_{g_\infty}([y_{i_s}],\frac r2)$, any
$g_\infty$-minimal geodesic from $z$ to $z'$ is contained in
$B_{g_\infty}([y_{i_s}],r)$.

\noindent (4.6) $d_{g_\infty}|_{B_{g_\infty}([y_{i_s}],\frac r2)}$
(resp. $d_X|_{B(x,\frac r2)}$) is determined by the lengths of
curves in $B_{g_\infty}([y_{i_s}],r)$ (resp. $B_{d_X}(x_{i_s},r)$).
The two length structures coincide, because $(K',g_k)\to
(K',g_\infty)$ in the $C^\infty$ sense.  As a consequence of (4.5)
and (4.6), we conclude that
$$f: \left(B_{d_X}\left(x_{i_s},\frac r2\right),d_X|_{B_{d_X}\left(x_{i_s},\frac r2\right)}\right)\to \left(B_{g_\infty}
\left([y_{i_s}],\frac r2\right),d_{g_\infty}|_{B_{d_{g_\infty}}
\left([y_{i_s}],\frac r2 \right)}\right)$$ is an isometry.

To uniquely extend $f: A\to (M\backslash E,g_\infty)$ to a
continuous surjection $f: (X,d_X)\to\overline{(M\backslash
E,d_{g_\infty})}$, one needs to show that $\{x_j\},\{y_j\}\subset A$
such that $d_X(x_j,y_j) \to 0$ implies that
$d_{g_\infty}(f(x_j),f(y_j))\to 0$; which may require that
$S_X\subset X$ has codimension at least $2$. Because we do not know
whether $\dim_{\mathcal{H}}(S_X)\le \dim_{\mathcal{H}}(X)-2$, we
will instead extend $f^{-1}: f(A)\to A$ to a continuous map,
$f^{-1}: \overline{(M\backslash E,d_{g_\infty})}\to (X,d_X)$. So, we
may assume that $\{x_j\}, \{y_j\}\subset f(A)$ such that
$d_{g_\infty}(x_j,y_j)\to 0$. Since $d_{g_\infty}$ is a length
metric, there is a path $\gamma_i\subset M\backslash E$ from $x_j$
to $y_j$ such that ${\rm
length}_{d_{g_\infty}}(\gamma_i)=d_{g_\infty}(x_j,y_j)+\delta_j$ and
$\delta_j\to 0$. Since $f^{-1}: (M\backslash E,d_{g_\infty})\to
(X\backslash S_X,d_X)$ is a local isometric embedding,
$$d_X(f^{-1}(x_j),f^{-1}(y_j))\le {\rm
length}_{d_{X}}(f(\gamma_i))={\rm
length}_{d_{g_\infty}}(\gamma_i)=d_{g_\infty} (x_j,y_j)+\delta_j\to
0.$$
\end{proof}

\vskip4mm

\section{Proofs  of   Theorem \ref{t1.1}, Theorem \ref{t1.01+} and Corollary \ref{c1.2}}
\vskip2mm

 Let  $M_{0}$ be a  Calabi-Yau $n$-variety with singular set $S$ which  admits a
crepant resolution $(\bar{M}, \bar{\pi})$,  and let
$\mathcal{L}_{0}$ be an ample line bundle on $M_{0}$. Note that
there is an embedding $M_{0}\hookrightarrow \mathbb{CP}^{N} $ such
that $\mathcal{L}_{0}^{m}=\mathcal{O}(1)|_{M_{0}}$ for an $m\geq 1$,
and that  the restriction of the Fubini-Study metric
$\omega_{FS}|_{M_{0}} $ represents $ m c_{1}(\mathcal{L}_{0}) $ in
$H^{1}(M_{0}, \mathcal{PH}_{M_{0}})$. By  Theorem 7.5 of \cite{EGZ},
there is a unique Ricci-flat K\"{a}hler metric $g$ on $M_{0} $ with
K\"{a}hler form $\omega \in  c_{1}(\mathcal{L}_{0})$. Let
$\{\bar{g}_{s}\}$ ($s\in (0, 1] $) be a family of Ricci-flat
  K\"{a}hler metrics with  K\"{a}hler classes
   $\lim\limits_{s\rightarrow 0}[\bar{\omega}_{s}]=\bar{\pi}^{*} c_{1}(\mathcal{L}_{0})$
   in $H^{1,1}(\bar{M}, \mathbb{R} )$, where $\bar{\omega}_{s}$ denotes
   the corresponding K\"{a}hler form of  $ \bar{g}_{s}$. Then
    \begin{equation}\label{e04.1}\lim_{s\rightarrow 0}\Vol_{\bar{g}_{s}}(\bar{M})=\frac{1}{n!}
 c_{1}^{n}(\mathcal{L}_{0})=\frac{1}{m^{n}n!}\int_{M_{0}}\omega_{FS}^{n}>0. \end{equation} Furthermore,
 it is  proved in \cite{To}  that
$$ \bar{g}_{s} \longrightarrow  \bar{\pi}^{*} g, \
 \  \ {\rm
and} \ \ \bar{\omega}_{s} \longrightarrow   \bar{\pi}^{*} \omega, \
\ \ {\rm when} \ \
 s\rightarrow 0,$$ in the $C^{\infty}$-sense  on any compact subset $K\subset\subset
 \bar{M}\backslash \bar{\pi}^{-1}(S)$.  By  \cite{RZ} and
 \cite{To}, the diameter  of $(\bar{M}, \bar{g}_{s})$ has a uniform
 bound, i.e.,   \begin{equation}\label{e04.2000+}{\rm diam}_{\bar{g}_{s}}(\bar{M})\leq C \end{equation} where
 $C$ is a constant
 independent of $s$.
By the Bishop-Gromov relative volume  comparison  and  (\ref{e04.1}),
$(\bar{M},\bar{g}_{s})$ is non-collapsed, i.e.,  there is a constant
$\kappa >0$ independent of $s$  such
  that  \begin{equation}\label{e04.2}
  {\rm Vol}_{\bar{g}_{s}}(B_{\bar{g}_{s}}(p,  r))\geq \kappa  r^{2n}, \end{equation} for any metric
  ball $B_{\bar{g}_{s}}(p,  r)\subset (M,\bar{g}_{s})$.  Gromov's  pre-compactness theorem (cf. \cite{G1})  implies that,
 for any sequence $s_{k}\rightarrow 0$, a subsequence of
 $(\bar{M},\bar{g}_{s_{k}})$ converges to a compact length metric space $(X, d_{X})$ in the
Gromov-Hausdorff topology.  First, we explore some metric properties
of $(X, d_{X})$.   \vskip2mm

\begin{lemma}\label{l3.1} Let   $(X, d_{X})$ be as in the above. Then  the following properties hold:
 \begin{itemize}\label{00}
  \item[i)]  there is a closed subset $S_{X}\subset X$ of Hausdorff
  dimension $\dim_{\mathcal{H}}S_{X} \leq 2n-4$, and $(X\backslash S_{X},
    d_{X}|_{X\backslash S_{X}})$ is a path metric space,
    i.e., for any
  $\delta>0$ and  any two points $x_{1}, x_{2}\in X\backslash S_{X}$,
  there is a cure $\gamma_{\delta}\subset X\backslash S_{X}$
  connecting $x_{1}$ and $ x_{2}$ satisfying $$
  \leng_{d_{X}}(\gamma_{\delta})\leq d_{X}(x_{1}, x_{2})+\delta,$$
   \item[ii)] there is a homeomorphic  local isometry $f: (X\backslash S_{X},d_X)\to (M_{0}\backslash
   S,d_{g})$, i.e.,    for  $x\in  X\backslash S_{X}$, there is
   an open subset
  $U_{x}\subset\subset X\backslash S_{X}$ such that for any $x_{1}$,  $x_{2}\in U_{x}$,
  $d_{X}(x_{1},x_{2})=d_{g}(f(x_{1}),f(x_{2})).$
   \item[iii)]  $(X, d_{X})$ is isometric to the metric completion
   $\overline{(M_{0}\backslash S,d_{g})} $.
   \end{itemize}
 \end{lemma}
 \vskip2mm

  \begin{proof}  Applying  general theorems in  \cite{CC1},   \cite{CT} and
  \cite{CCT} to our situation, i.e.,  $(\bar{M}, \bar{g}_{s_{k}})
\stackrel{d_{GH}}\longrightarrow (X, d_{X})$,   we see the following
properties:
\begin{itemize}\label{00}
  \item[i)]   there is a closed subset $S'\subset X$ of Hausdorff
  dimension $\dim_{\mathcal{H}}S' \leq 2n-4$ such  that
  for any $x\in S'$, there is a tangent cone $T_{x}X$ that is    not
  isometric to $\mathbb{R}^{2n} $.
  \item[ii)] $X\backslash S'$ is a smooth open  complex manifold,  and $d_{X}|_{X\backslash S'}$ is induced
  by a Ricci-flat K\"{a}hler metric $g_{\infty}$ on $X\backslash S'
  $.\end{itemize}
    From   Section 3 of  \cite{CC2}, we see that  for any $x_{1}, x_{2}\in X\backslash S'$,
    and any $\delta >0$, there is  a  curve $\gamma_{\delta}$ connecting $x_{1}$ and $
 x_{2}$ in  $X\backslash S'$ such that $$ {\rm length}_{d_{X}}(\gamma_{\delta}) \leq
  \delta
  +d_{X}(x_{1},x_{2}).$$

  Note that $\bar{\pi}^{-1}(S)$ is a  finite disjoint
   union of complex subvarieties  $E_{i}$, i.e.,  $\bar{\pi}^{-1}(S)=\coprod E_{i}$.
   If  $S_X\subset X$ denotes  the subset consisting of points $x\in X$
such that for each $k$  there is an   $\bar{x}_k$ in  the smooth
part of $ \bar{\pi}^{-1}(S) \subset (\bar{M},\bar{g}_{s_{k}})$ and
  $\bar{x}_k\to x$ under the Gromov-Hausdorff convergence of
$\{(\bar{M},\bar{g}_{s_{k}})\}$  to $(X, d_{X})$, then  by Theorem
\ref{t2.01}  there is a homeomorphic  local isometry $f:
(X\backslash S_{X},d_X) \rightarrow (M_{0}\backslash S,g)$.  Thus,
for any $x\in X\backslash S_{X}$, the tangent cone $T_{x}X $ is
unique and  isometric to $\mathbb{R}^{2n} $, which implies that
$X\backslash S_{X}\subseteq X\backslash S' $, i.e.,  $ S'\subseteq
S_{X}$.

 We claim that  $S_{X}=  S' $.
If false, there is a  $x\in S_{X}\backslash S'$ and there is a
$\sigma>0$ such  that the metric ball $ B_{g_{\infty}}
(x,\sigma)\subset X\backslash S'$.
 By  the volume
    convergence theorem due  to  Cheeger and Colding  (cf. \cite{C},
    \cite{CC1}) and from $\bar{x}_{k}\rightarrow
x $,  we derive that for any $0< \rho\leq \sigma$,
    $$\lim_{k\rightarrow\infty}{\rm Vol}_{\bar{g}_{s_{k}}}(B_{\bar{g}_{s_{k}}}(\bar{x}_{k},\rho))
    ={\rm
    Vol}_{g_{\infty}}(B_{g_{\infty}}(x,\rho)).$$  Since $g_{\infty}$
    is a smooth metric, $\lim\limits_{\rho\rightarrow 0}|
    \frac{{\rm
    Vol}_{g_{\infty}}(B_{g_{\infty}}(x,\rho))}{\varpi_{2n}\rho^{2n}}-1|=0$,
    where $\varpi_{2n}$ denotes the volume of the  metric 1-ball in the
     Euclidean space $\mathbb{R}^{2n}$.  Thus for any
    $\varepsilon>0$ we can find a  $\rho \ll 1$ and a  $k(\rho)\gg 1$ such
    that for any $k\geq k(\rho)$ we have  $$\left|
    \frac{{\rm Vol}_{\bar{g}_{s_{k}}}(B_{\bar{g}_{s_{k}}}(\bar{x}_{k},\rho))
    }{\varpi_{2n}\rho^{2n}}-1\right|\leq \varepsilon.$$ By the proof of  Theorem
    3.2 in \cite{An2},  we see that  there is a uniform  lower bound $0< \rho_{h} < \rho$
     (independent of $s_{k}$)  for  the
    harmonic radius  of $\bar{g}_{s_{k}}$ at  $\bar{x}_{k}$ i.e.,
        there are harmonic coordinates $
    h^{1}, \cdots, h^{2n}$ on
    $B_{\bar{g}_{s_{k}}}(\bar{x}_{k},\rho_{h})$ such that  $\bar{g}_{s_{k}}
    =\sum\limits_{ij} \bar{g}_{s_{k},ij}dh^{i}dh^{j}$, $$2^{-1}(\delta_{ij})\leq ( \bar{g}_{s_{k},ij})
     \leq 2 ( \delta_{ij}),
     \ \ \ \ \ \ \rho^{1+\alpha}\| \bar{g}_{s_{k},ij}\|_{C^{1,\alpha}}\leq2,
     $$ where $\alpha\in(0,1) $. Furthermore, by Ricci flatness
     there are constants $C_{l}>0$ independent of $k$  such that $$\| \bar{g}_{s_{k},ij}\|_{C^{l}}\leq
    C_{l}  $$ on  $B_{\bar{g}_{s_{k}}}(\bar{x}_{k},
\frac{\rho_{h}}{2})$ (cf. Section 4  in   \cite{An0}). Hence   the
sectional curvature ${\rm Sec}_{\bar{g}_{s_{k}}} $ of
   $\bar{g}_{s_{k}}$ on $B_{\bar{g}_{s_{k}}}(\bar{x}_{k},
\frac{\rho_{h}}{2})$ and the injectivity radius
$i_{\bar{g}_{s_{k}}}(\bar{x}_{k})$ have uniform bounds,
$$\sup_{B_{\bar{g}_{s_{k}}}(\bar{x}_{k},
\frac{\rho_{h}}{2})}|{\rm Sec}_{\bar{g}_{s_{k}}} |\leq \Lambda, \ \
\
 \ \ \ i_{\bar{g}_{s_{k}}}(\bar{x}_{k}) >\iota,$$ where $\Lambda$ and $\iota$ are
  two
constants  independent of $k$.

 In the rest of proof of Lemma \ref{l3.1},   we need the
following theorem.
 \vskip2mm
\begin{theorem}\label{p2+.1}
Let $(M,  g, \omega)$ be   a complete   K\"{a}hler   $n$-manifold,
    and $p\in M$. Assume that the sectional curvature ${\rm Sec}_{g}$
  satisfies $$ \sup_{B_{g}(p, \frac{2\pi}{\sqrt{\Lambda}})}{\rm Sec}_{g}\leq \Lambda, \ \ \ \ \Lambda>0, $$ and there is a
   complex subvariety  $E$ of dimension $m\leq n$ such that  $p$ belongs
   to the regular part of $
  E$.  Then  $${\Vol}_{g}(B_{g}(p,r)\cap E)\geq \varpi r^{2m},
  $$ for any $r\leq \min \{i_{g}(p), \frac{\pi}{2\sqrt{\Lambda}}\}$, where $i_{g}(p)$ denotes the injectivity radius
  of $g$ at $p$, and $\varpi=\varpi(m,\Lambda)$ is a constant depending only on $m$ and
  $\Lambda$.
\end{theorem}
 \vskip2mm
Note that  similar volume  comparison results  were  obtained  for
smooth minimal
 submanifolds in  \cite{Ma} and  \cite{Go}, for  complex subvarieties  of
 $\mathbb{C}^{n}$ in \cite{GH1}, and for  minimal currents  in  $\mathbb{R}^{n}$
  (cf.  \cite{Mo}).  Since the authors could not find a
  proof of Theorem \ref{p2+.1} in the
literature, we shall  present a proof  at the end of this
  section.

  By Theorem
\ref{p2+.1}  and  taking $r=\min \{\iota,
\frac{\pi}{2\sqrt{\Lambda}}, \frac{\rho_{h}}{2}\}$, we obtain
$$\Vol_{\bar{g}_{s_{k}}}(\bar{\pi}^{-1}(S) )\geq  \Vol_{\bar{g}_{s_{k}}}(\bar{\pi}^{-1}(S) \cap
B_{\bar{g}_{s_{k}}}
  (\bar{x}_{k},r))>C,$$where   $C>0$ is  a constant independent of $k$. On
  the other hand, since $\lim\limits_{s\rightarrow 0}[\bar{\omega}_{s}]=\bar{\pi}^{*} c_{1}(\mathcal{L})|_{M_{0}} $
   in $H^{1,1}(M, \mathbb{R} )$, we have $$\lim_{k\rightarrow\infty}
   \Vol_{\bar{g}_{s_{k}}}(\bar{\pi}^{-1}(S) )=  \sum_{i}\lim_{k\rightarrow\infty} \frac{1}{\dim_{\mathbb{C}} E_{i}!}\int_{E_{i} }
  \bar{\omega}_{s_{k}}^{\dim_{\mathbb{C}} E_{i}}=0,  $$ which is a contradiction.

   Note that by i)  $(X\backslash S_{X}, d_{X})$ coincides with the length
   metric structure $(X\backslash S_{X}, d_{d_{X}})$. Consequently,
   $f: (X\backslash S_{X},d_X)\to (M_{0}\backslash
   S,d_{g})$ is an isometry, and this implies  iii).
    \end{proof}
\vskip2mm

    \begin{lemma}\label{l3.2}Let   $(X, d_{X})$ be as in Lemma \ref{l3.1}, and let $(M_{k}, g_{k}, \omega_{k})$ be any
     family of
    Ricci-flat K\"{a}hler  $n$-dimensional manifolds satisfying
 \begin{itemize}\label{00}
  \item[i)] $$ \lim_{k \rightarrow \infty}\Vol_{g_{k}}(M_{k})=\frac{1}{n!}
 c_{1}^{n}(\mathcal{L}_{0}), $$
   \item[ii)]  there is a family of embeddings $F_{k}: M_{0}\backslash S \rightarrow
M_{k}$ such that $$F_{k}^{*}g_{k} \rightarrow g, \ \ {\rm and} \ \
F_{k}^{*}\omega_{k} \rightarrow \omega, \ \ \ {\rm when} \ \ k
\rightarrow \infty,$$ in the $C^{\infty}$-sense on any compact
subset $K\subset\subset M_{0}\backslash S$.
   \end{itemize} Then  $$( M_{k}, g_{k}) \stackrel{d_{GH}}\longrightarrow (X,
d_{X}) \stackrel{d_{GH}}\longleftarrow ( \bar{M}, \bar{g}_{s_{k}}).
$$
 \end{lemma}

\vskip2mm
 \begin{proof}    By  Lemma \ref{l3.1},  there is a homeomorphic
  local isometry $f: (X\backslash S_{X},d_X)\to (M_{0}\backslash
   S,g)$.
  For an  $x\in X\backslash S_{X} $,   ii) of Lemma \ref{l3.2}
  implies that $ {\rm Vol}_{g_{k}}(B_{g_{k}}(F_{k}(f(x)), 1))\geq \upsilon
  $ for a constant $\upsilon >0 $ independent of $k$. For any
  $1<R< {\rm
diam}_{g_{k}}(M_{k})$,  Lemma \ref{t3.03}  (Lemma 2.3 in \cite{Pa})
shows that
$$R\leq 1+ 2n \frac{{\rm Vol}_{g_{k}}(B_{g_{k}}(F_{k}(f(x)),
2R+2))}{{\rm Vol}_{g_{k}}(B_{g_{k}}(F_{k}(f(x)), 1))}.$$ By taking
$R=\frac{1}{2} {\rm diam}_{g_{k}}(M_{k})$, we obtain that
$${\rm diam}_{g_{k}}(M_{k})<2+ 4n\upsilon^{-1}\frac{1}{n!}
 c_{1}^{n}(\mathcal{L}_{0}).$$   By the
 Bishop-Gromov relative volume   comparison, $(M_{k},g_{k})$ is
non-collapsed, i.e.,  there is a constant $\kappa >0$ independent of
$k$  such
  that  for any metric
  ball $B_{g_{k}}(p,  r) \subset M_{k} $  \begin{equation}\label{e034.1}
  {\rm Vol}_{g_{k}}(B_{g_{k}}(p,  r))\geq \kappa  r^{2n}.   \end{equation}
  Gromov's  pre-compactness   theorem implies that  a
  subsequence of    $\{(M_{k}, g_{k})\}$ $d_{GH}$-converges to a compact length  metric
  space $(Y, d_{Y})$.  Following   the proof of Theorem 4.1 in  \cite{RZ} with a
  minor modification
   will  prove  that $(Y, d_{Y})$ is isometric to $(X, d_{X})$.
    Because of this,    we only  present a  sketch of the proof.

First, the   same arguments  as the proof of Lemma 4.1 in \cite{RZ}
imply   that there exists an
  embedding $\psi': (M_{0}\backslash S, g)\rightarrow (Y, d_{Y})$ which is a local
  isometry. Hence $\psi=\psi'\circ f: (X\backslash S_{X}, d_{X})\rightarrow  (Y,
  d_{Y})$ is a local isometric embedding.
  Thus, if $\gamma$ is a geodesic in $ ( X\backslash S_{X},d_{X})$,
  then $\psi(\gamma)$ is a geodesic in $ ( \psi(X\backslash
  S_{X}),d_{Y})$.
  For any  $x_{1}, x_{2}\in X$, and  two
      sequences  $\{x_{1,j}\}, \{x_{2,j}\}\subset X\backslash S_{X}$   converging to $x_{1}$ and $ x_{2}$ respectively,
   there are curves    $\gamma_{j}$  connecting
      $x_{1,j}$ and  $x_{2,j}$ in $X\backslash S_{X}$ with ${\rm length}_{
      g}(\gamma_{j})\leq
      d_{X}(x_{1,j}, x_{2,j})+\frac{1}{j}$ by Lemma \ref{l3.1},  which implies
         \begin{equation}\label{e034.12}d_{Y}(\psi(x_{1,j}), \psi(x_{2,j}))\leq
      {\rm length}_{d_{Y}}(\psi(\gamma_{j}))= {\rm length}_{
      g}(\gamma_{j})\leq
      d_{X}(x_{1,j}, x_{2,j})+\frac{1}{j}.  \end{equation}   If $x_{1}=x_{2}=x$,   both
      $\{\psi(x_{j})\}$ and $\{\psi(x'_{j})\}$ are Cauchy sequences, and converge to  the same  limit
       $y$ in $Y$.  By defining  $\tilde{\psi}(x)=y$, $\psi$ extends to a  continuous
    map
      $\tilde{\psi}: X \longrightarrow Y$ such that   $\tilde{\psi}(X)\subseteq  Y$ is closed.

   If $\tilde{\psi}(X)\subsetneq Y$, then  there is  a metric ball   $B_{d_{Y}}(y, \delta)\subset\subset Y\backslash
    \tilde{\psi}(X)$ for a $\delta >0$.  By (\ref{e04.1}), (\ref{e034.1}), and  the volume
    convergence theorem due  to  Cheeger and Colding  (cf. \cite{C},
    \cite{CC1}), we derive
  $$
   \mathcal{H}^{2n}(Y)=\mathcal{H}^{2n}(X)={\rm Vol}_{g}(X\backslash S_{X})
   \quad \mbox{ and } \quad \mathcal{H}^{2n}(B_{d_{Y}}(y, \delta))\geq \kappa \delta^{2n} , $$ where
 $\mathcal{H}^{2n} $ denotes    the $2n$-dimensional Hausdorff measure. Thus  $$\mathcal{H}^{2n}(Y)\geq
   \mathcal{H}^{2n}(\psi(X\backslash S_{X}))+\mathcal{H}^{2n}(B_{d_{Y}}(y,
   \delta))\geq {\rm Vol}_{g}(X\backslash S_{X})+ \kappa \delta^{2n}>\mathcal{H}^{2n}(Y),$$
   a contradiction.

   To show that $\tilde{\psi}$ is an isometry, we first check that
   $\tilde{\psi}$ is 1-Lipschitz.
For any $x_{1}\neq  x_{2}\in X$, there are
  sequences of points $\{x_{i,j}\}\subset X\backslash S_{X}$, $i=1,2$,  such that $d_{X}(x_{i,j}, x_{i})\rightarrow 0$ when $j\rightarrow\infty$.
  Thus $d_{Y}(\psi(x_{i,j}), \tilde{\psi}( x_{j}))\rightarrow 0$, $i=1,2$,  when
  $j\rightarrow\infty$.  By (\ref{e034.12}) and letting $j\rightarrow\infty$,
we obtain that \begin{equation}\label{3.8}
d_{Y}(\tilde{\psi}(x_{1}),\tilde{\psi}(x_{2}))  \leq
  d_{X}(x_{1}, x_{2}) \end{equation} i.e.,  $\tilde{\psi}$ is a 1-Lipschitz  map.

   Because   $\tilde{\psi}(S_{X})\bigcup \psi(X\backslash S_{X})=Y$,
    $\tilde{\psi}(S_{X})\supseteq Y\backslash \psi(X\backslash S_{X})$.
    Since
  $\dim_{\mathcal{H}}S_{X} \leq 2n-2$,
 $$\mathcal{H}^{2n-1}(Y\backslash \psi(X\backslash S_{X}))\leq \mathcal{H}^{2n-1}( \tilde{\psi}(S_{X}) )\leq
\mathcal{H}^{2n-1}(S_{X})=0.
 $$   If there is a $j$ such that  $\varrho=d_{X}(x_{1,j},
x_{2,j})-d_{Y}(\psi(x_{1,j}), \psi(x_{2,j}))>0$,
 then by  Section 3 of  \cite{CC2} we see
   that
there
  is a curve $\bar{\gamma}$ connecting $\psi(x_{1,j}), \psi(x_{2,j})$ in
    $ \psi(X\backslash S_{X})$ and  $$
   d_{X}(x_{1,j}, x_{2,j})\leq   {\rm
    length}_{d_{X}}(\psi^{-1}(\bar{\gamma}))=
    {\rm length}_{d_{Y}}(\bar{\gamma})\leq d_{Y}(\psi(x_{1,j}), \psi(x_{2,j}))+
    \frac{1}{2}\varrho,$$  a contradiction. Then $d_{X}(x_{1,j},
x_{2,j})=d_{Y}(\psi(x_{1,j}), \psi(x_{2,j}))$ and thus   by letting
  $j\rightarrow\infty$, we obtain that  $$d_{Y}(\tilde{\psi}(x_{1}),\tilde{\psi}(x_{2}))
  =
  d_{X}(x_{1}, x_{2}).$$ By now, we have proved that   $\tilde{\psi}:(X,d_{X})\longrightarrow (Y,d_{Y})$  is an
  isometry.
 \end{proof}

\vskip4mm

After the above preparation, we are ready to prove Theorem
\ref{t1.1}, Theorem \ref{t1.01+} and Corollary \ref{c1.2}. \vskip2mm

 \begin{proof}[Proof of Theorem \ref{t1.1}]
Let $M_{0}$, $\pi: \mathcal{M}\rightarrow \Delta$,
$\mathcal{K}_{\mathcal{M}/\Delta}$, $\mathcal{L}$, $M_{t}$,
$\tilde{g}_{t}$, $\tilde{\omega}_{t}$, $(\bar{M}, \bar{\pi})$,
$\bar{g}_{s}$, $\bar{\omega}_{s}$ be as in Theorem \ref{t1.1}.

 Note that  $$\lim_{s\rightarrow 0}\Vol_{\bar{g}_{s}}(\bar{M})=\frac{1}{n!}
 c_{1}^{n}(\mathcal{L})|_{M_{0}} \equiv \frac{1}{n!}
c_{1}^{n}(\mathcal{L})|_{M_{t}}= \Vol_{\tilde{g}_{t}}(M_{t}) .
$$ By   \cite{To},
$$   \bar{g}_{s} \longrightarrow  \bar{\pi}^{*} g, \
 \  \ {\rm
and} \ \ \bar{\omega}_{s} \longrightarrow   \bar{\pi}^{*} \omega, \
\ \ {\rm when} \ \
 s\rightarrow 0,$$ in the $C^{\infty}$-sense  on any compact subset $K\subset\subset
 \bar{M}\backslash \bar{\pi}^{-1}(S)$.  By (\ref{e04.2000+}) and Ricci flatness,   we apply Gromov's compactness theorem
 to conclude  that
 for any sequence $s_{k}\rightarrow 0$, a subsequence of
 $(\bar{M},\bar{g}_{s_{k}})$ $d_{GH}$-converges to a compact path metric space $(X, d_{X})$,
 which satisfies the conclusion of Lemma
\ref{l3.1}.   By  Lemma \ref{l3.2},   for any other sequence
$s'_{k}\rightarrow 0$,
 $(\bar{M},\bar{g}_{s'_{k}})$ $d_{GH}$-converges to  $(X, d_{X})$ too.
  Thus $$\lim_{s\rightarrow 0}d_{GH}
((\bar{M},\bar{g}_{s}), (X, d_{X}))=0. $$

By  Theorem \ref{t4.2+},  $$F_{t}^{*}\tilde{g}_{t} \rightarrow g, \
\ {\rm and} \ \ F_{t}^{*}\tilde{\omega}_{t} \rightarrow \omega, \ \
\ {\rm when} \ \ t\rightarrow 0,$$ in the $C^{\infty}$-sense on any
compact subset $K\subset\subset M_{0}\backslash S$, where $F_{t}:
M_{0}\backslash S \rightarrow M_{t}$ is a family of embeddings.
 By   Lemma \ref{l3.2} and the fact that the limit is independent of convergent
  subsequences,   we obtain   the
 conclusion,   $$\lim_{t\rightarrow 0}d_{GH}
((M_{t},\tilde{g}_{t}), (X, d_{X}))=0. $$
\end{proof}

\vskip2mm

 The same argument in the proof of Theorem  \ref{t1.1} also
gives a proof of Theorem \ref{t1.01+}.

\vskip2mm
 \begin{proof}[Proof of Corollary  \ref{c1.2}] Let   $M$ be  a Calabi-Yau
 manifold, and  let $\omega_{s}$ ($s\in [0,1]$)  be a family of Ricci-flat
 K\"{a}hler forms.  It is clear that $\omega_{s}$ converges to
 $\omega_{0}$ when  $s\rightarrow 0$ in the $C^{\infty}$-sense, which implies  that $\mathfrak{M}_{M}$ is path
 connected in $(\mathfrak{X},d_{GH})$. Let  $\pi': \bar{\mathcal{M}}\rightarrow \Delta$ be  a
 smooth family of Calabi-Yau
 manifolds over the unit  disc $\Delta\subset \mathbb{C} $ with
 an ample line bundle  $ \mathcal{L}$ on $\bar{\mathcal{M}}  $, and let  $ \tilde{\omega}_{t}$ be the
 unique Ricci-flat
 K\"{a}hler form on $\pi'^{-1}(t)=M_{t} $ with $ \tilde{\omega}_{t}\in  c_{1}(
 \mathcal{L})|_{M_{t}}$. It is standard   that $
F_{t}^{*} \tilde{\omega}_{t}$ converges to $ \tilde{\omega}_{0}$ in
the $C^{\infty}$-sense,  when $t\rightarrow 0
 $, where $F_{t}:M_{0}\rightarrow M_{t}$ is a smooth family of
 diffeomorphisms.  Thus, if $\pi': \bar{\mathcal{M}}\rightarrow \mathcal{D}$ is   a
 smooth family of Calabi-Yau
 manifolds over connected  complex manifold  $\mathcal{D}$,  then
  $$\bigcup_{M_{t}=\pi'^{-1}(t),  t\in \mathcal{D}} \mathfrak{M}_{M_{t}}
 \subset (\mathfrak{X},d_{GH}) $$  is path connected.

 Note  that
 Calabi-Yau manifolds are minimal models.
If $M$ and $M'$ are  two birationally equivalent three-dimensional Calabi-Yau
manifolds, then $M$ and $M'$ are related by a sequence of flops
(cf.  \cite{Kol}  \cite{KM}), i.e., there is  a sequence of
varieties $ M_{1}, \cdots, M_{k}$ such that $M=M_{1}$, $M'=M_{k}$,
and $M_{j+1}$ is obtained by a flop from $M_{j}$. Consequently there
are normal  projective varieties $ M_{0,1}, \cdots, M_{0,k-1}$, and
small resolutions $\bar{\pi}_{j}:M_{j}\rightarrow M_{0,j}$ and
$\bar{\pi}_{j}^{+}:M_{j}\rightarrow M_{0,j-1}$.  By  \cite{Kol},
$M_{j}$ has  the same singularities as $M$, and thus $M_{j}$ is
smooth.  Since the exceptional locus of $\bar{\pi}_{j}$ and
$\bar{\pi}_{j}^{+}$ are of co-dimension at least   2,  $M_{0,j}$ has
only canonical singularities, and  the canonical bundle  of
$M_{0,j}$ is  trivial (cf. Corollary 1.5 in  \cite{Kaw}).  Therefore
$M_{0,j}$ is a three-dimensional Calabi-Yau variety, and $M_{j}$ is a
three-dimensional Calabi-Yau
manifold. By  Theorem \ref{t1.01+}, for any $j>0$,
$$\overline{\mathfrak{M}}_{M_{j}}\bigcup
\overline{\mathfrak{M}}_{M_{j+1}}$$ is path connected,  where
$\overline{\mathfrak{M}}_{M_{j}}$ denotes the closure of
$\mathfrak{M}_{M_{j}}\subset (\mathfrak{X}, d_{GH})$.  By now we
have proved  i) of
 Corollary \ref{c1.2}.

Let  $M_{0}$ be a three-dimensional complete  intersection  Calabi-Yau
conifold    in $\mathbb{C}P^{m_{1}}\times\cdots\times
 \mathbb{C}P^{m_{l}}$, and $\bar{M}$ be a small resolution of  $M_{0}$,
 which is a three-dimensional complete  intersection  Calabi-Yau (CICY)
manifold   in  products  of  projective
 spaces. By  Theorem \ref{t1.1}, we see
  that  $$\bigcup_{\tilde{M}\in \mathfrak{D}(M_{0})}\overline{\mathfrak{M}}_{\tilde{M}}\bigcup
\overline{\mathfrak{M}}_{\bar{M}}\subset (\mathfrak{X}, d_{GH})$$ is
path connected, where
 $\mathfrak{D}(M_{0}) $ denotes the set of three-dimensional  CICY
manifolds   in $\mathbb{C}P^{m_{1}}\times\cdots\times
 \mathbb{C}P^{m_{l}}$ obtained by a  smoothing  of $M_{0}$. If $M$ and $M'$ are  two three-dimensional
 CICY
manifolds   in  products  of  projective
 spaces, then  by   \cite{GH}  $M$ and $M'$ are related by a sequence of conifold
 transitions, or  inverse conifold transitions.  Precisely,  there is  a sequence
 of three-dimensional CICY
manifolds $ M_{1}, \cdots, M_{k}$ with $M=M_{1}$ and $M'=M_{k}$
such that  for any $1\leq j \leq k$, there is  a three-dimensional CICY
conifold  $M_{0,j}$  in some $\mathbb{C}P^{m_{1}}\times\cdots\times
 \mathbb{C}P^{m_{l}}$, $M_{j}$ is a small resolution of $M_{0,j}$
 and $M_{j+1}\in \mathfrak{D}(M_{0,j})$, or vice  versa   $M_{j+1}$ is a small resolution of $M_{0,j}$
 and $M_{j}\in \mathfrak{D}(M_{0,j})$.  Thus  ii) of
 Corollary \ref{c1.2} is followed i.e., $\overline{\mathfrak{CP}}$
 is path connected.

 In  \cite{ACJM} and  \cite{CGGK},    many  complete
intersection  Calabi-Yau
   3-manifolds in toric varieties  were verified  to be
   connected
   by extremal transitions, which include
   Calabi-Yau hypersurfaces in  all  toric 4-manifolds obtained by resolving
    weighted projective 4-spaces. Let $\mathfrak{CT}_{0}$ be the set
    of the above Calabi-Yau 3-manifolds with Ricci-flat K\"ahler metrics of volume
    1. By Theorem \ref{t1.1}, the closure
    $\overline{\mathfrak{CT}}_{0}$ is path connected. Note that   a quintic in $\mathbb{CP}^{4}$ with
    Ricci-flat K\"ahler metric of volume
    1  is in  $\mathfrak{CT}_{0}\bigcap \mathfrak{CP} $.
   Thus  iii) of
 Corollary \ref{c1.2} is obtained.
 \end{proof}

 \vskip2mm
Now we give a proof of Theorem  \ref{p2+.1}, which can be viewed as
a combination of the proof of   Theorem 2.0.1   in \cite{Go} and the
proof of Theorem 9.3 in \cite{Mo}.

\vskip2mm
  \begin{proof}[ Proof of Theorem \ref{p2+.1}] For any $r\leq \min \{i_{g}(p), \frac{\pi}{2\sqrt{\Lambda}}\}$,  there are normal
  coordinates such that $g=d\rho^{2}+g_{\rho}$  on $B_{g}(p,r)$ where $g_{\rho}$ is a Riemannian metric on $ \partial
  B_{g}(p,\rho)$, $0< \rho\leq r$. Let  $f(q, \rho)=\exp_{p} \frac{\rho}{r}(\exp^{-1}_{p}
  q)$ for any $q\in \partial
  B_{g}(p,r)$.   Then $f:\partial
  B_{g}(p,r)\times (0,r] \rightarrow B_{g}(p,r)\backslash \{p\}$ is
  a diffeomorphism.  For any $w\in T_{q}(\partial
  B_{g}(p,r))$, it is clear that   $J(\rho)=df|_{(\rho, q)}w$ is a  normal  Jacobi field along the geodesic
    $\gamma(\rho)=f(q, \rho)$ with $J(0)=0
  $ and $J(r)=w$.
 A standard Rauch comparison  argument  shows that
 $$|J(\rho)|_{g_{\rho}}\leq \frac{\sin \sqrt{\Lambda}\rho}{\sin \sqrt{\Lambda}
 r}|J(r)|_{g_{r}} =
   \frac{\sin \sqrt{\Lambda}\rho}{\sin \sqrt{\Lambda} r}
   |w|_{g_{r}}, \ \ $$ (cf. Lemma 2.0.1  \cite{Go}). Thus the norm  of  the
   differential  $df|_{\partial
  B_{g}(p,r)\times\{\rho\}}$ corresponding the  metric
   $g_{r}$ on $\partial
  B_{g}(p,r)\times\{\rho\}$ and $g_{\rho}$ on $\partial
  B_{g}(p,\rho)$ satisfies
  $$|df|_{\partial
  B_{g}(p,r)\times\{\rho\}}|\leq  \frac{\sin \sqrt{\Lambda}\rho}{\sin \sqrt{\Lambda}
  r},$$
    which  implies
    \begin{equation}\label{e++1}   g_{\rho}\leq \frac{\sin^{2} \sqrt{\Lambda}\rho}{\sin^{2} \sqrt{\Lambda}
  r}g_{r}.  \end{equation}

    Denote  $$\Theta(r)={\Vol}_{g}(B_{g}(p,r)\cap
  E).$$ Since $\Theta(r)$ is monotonically increasing, $\Theta'(r)$ exists for
  almost all $r$.   By 4.11(3) in \cite{Mo},  we have
   $$\mathcal{H}_{g_{r}}^{2m-1}(\partial B_{g}(p,r)\cap
  E)\leq \Theta'(r). $$ Let  $\mathcal{C}=(\partial
  B_{g}(p,r)\cap E)\times  (0,r]\subset \partial
  B_{g}(p,r)\times  (0,r]=B_{g}(p,r)\backslash \{p\}$. By  Fubini's
  theorem and (\ref{e++1}), we see  that \begin{eqnarray*}
  \mathcal{H}_{g}^{2m}(\mathcal{C})& = &
  \int_{0}^{r}\mathcal{H}_{g_{\rho}}^{2m-1}(\partial B_{g}(p,r)\cap
  E)d\rho \\ & \leq & \mathcal{H}_{g_{r}}^{2m-1}(\partial B_{g}(p,r)\cap
  E)\int_{0}^{r}\frac{\sin^{2m-1} \sqrt{\Lambda}\rho}{\sin^{2m-1} \sqrt{\Lambda}
  r}d\rho . \end{eqnarray*} Since $E$ is a complex subvariety, $E$ is a volume minimizer and thus
    $$\Theta(r)\leq
  \mathcal{H}_{g}^{2m}(\mathcal{C}) \leq \frac{\int_{0}^{r} \sin^{2m-1} \sqrt{\Lambda}\rho d\rho }{\sin^{2m-1} \sqrt{\Lambda}
  r} \Theta'(r).$$ Therefore  $$  \frac{d}{dr}\left(\frac{\Theta(r)}
  {\int_{0}^{r} \sin^{2m-1} \sqrt{\Lambda}\rho
  d\rho}\right)\geq 0.$$ Since $p$ is a smooth point of $E$, $$\lim_{\bar{r}\rightarrow0}
  \frac{\Theta(\bar{r})}{\int_{0}^{\bar{r}} \sin^{2m-1} \sqrt{\Lambda}\rho
  d\rho}=C,  $$ where $C$
  is a constant depending only on $\Lambda $ and $m$. Thus
   $$\Theta(r)\geq C\int_{0}^{r} \sin^{2m-1} \sqrt{\Lambda}\rho
  d\rho \geq \varpi r^{2m} .$$
    \end{proof}
\vskip5mm

\vspace{0.7cm}

\appendix
\section{Gromov-Hausdorff Convergence of Compact Metric Spaces}

\vskip2mm

In the proof of Theorem \ref{t2.01},  we freely used some basic
properties of the Gromov-Hausdorff convergence of compact metric
spaces. For the convenience of readers, we will briefly recall
related notions and proofs of these properties (cf. \cite{Rong}).

Let $(Z,d)$ be a metric space, and let $C^Z$ denote the set of all
compact subsets of $Z$. For $A, B\in C^Z$, the  Hausdorff distance
of $A$ and $B$  is
$$d_H(A,B)=\inf\{\epsilon, U_\epsilon(A)\supseteq B, U_\epsilon(B)\supseteq A\},$$
where $U_\epsilon(S)$ denotes the $\epsilon$-neighborhood of $S$.
Then $(C^Z,d_H)$ is a complete metric space. The Gromov-Hausdorff
distance can be viewed as an abstract extension of $d_H$ on $
\mathfrak{X}$: the space of isometric classes of all compact metric
spaces. For $X, Y\in \mathfrak{X}$, the Gromov-Hausdorff distance of
$X$ and $Y$ is
$$d_{GH}(X,Y)=\inf_Z\{d_H^Z(X,Y), \exists \text{ isometric embeddings,
 $X, Y\hookrightarrow Z$, a metric space}\}.$$ In the above
 definition, one can consider the disjoint union  that  $Z=X\coprod
 Y$ with an admissible metric $d$, i.e., a metric on $Z$ such that
 the restriction on $X$ (resp. $Y$) is the metric on $X$ (resp.
 $Y$).

It is not hard to check that $d_{GH}(X,Y)=0$  if and only if $X$ is
isometric to $Y$ and $d_{GH}$ satisfies the triangle inequality.
Hence, $(  \mathfrak{X} ,d_{GH})$ is a metric space.

In the proof of Theorem \ref{t2.01}, the following proposition is
used.

\vskip2mm

\begin{proposition}\label{A1}

Given $\{X_i\}$ in $  \mathfrak{X} $ such that $d_{GH}(X_i,X_{i+k})
<2^{-i}$ for all $i$ and $k$, let $Y=\coprod\limits_{i} X_i$.

\noindent i) There is a  metric $d_Y$ on $Y$ such that the
restriction of $d_Y$ on each $X_i$ is the metric on $X_i$ and
$\{X_i\}$ is a Cauchy sequence with respect to $d_{Y,H}$.

\noindent ii) Let $X$ be the collection of equivalent Cauchy
sequences, $\{\{x_i\}, \,\, x_i\in X_i\}$, equipped with the metric
$\hat d(\{x_i\}, \{y_i\})= \lim\limits_{i\rightarrow \infty}
d_Y(x_i,y_i)$. Then $Y\coprod X$ has an admissible metric defined by
$d(x,\{x_i\})= \lim\limits_{i\rightarrow \infty} d(x,x_i)$.

\noindent iii) For all $\epsilon>0$, $X$ has a finite
$\epsilon$-dense subset (thus the completion of $X$ is compact).

\noindent iv) $d_H(X_i,X)\to 0$ as $i\to \infty$.
\end{proposition}

\vskip2mm

\begin{proof} i) We first take, for each $i$, an admissible
metric $d_{i,i+1}$ on $X_i\coprod X_{i+1}$ such that $d_{i,i+1}
(X_i,X_{i+1})<d_{GH}(X_i,X_{i+1})+2^{-i}<2^{-i+1}$. We then extend
$\{d_{i,i+1}\}$ to an admissible metric on $Y$ by defining, for each
pair $(i,j)$, an admissible metric $d_{i,i+j}$ on $X_i\coprod
X_{i+j}$, as follows:
$$d_Y(x_i,x_{i+j})=\inf _{x_{i+k}\in X_{i+k}}\{\sum_{k=0}^{j-1}
d_{i+k,i+k+1}(x_{i+k},x_{i+k+1})\}.$$ It is straightforward to check
that $d_Y$ satisfies the triangle inequality. Then $\{X_i\}$ is a
Cauchy sequence with respect to $d_{Y,H}$, because for all $j$,
\begin{eqnarray*} d_{Y,H}(X_i,X_{i+j})&\le & d_{Y,H}(X_i,X_{i+1})+\cdots +d_{Y,H}
(X_{i+j-1},X_{i+j})\\&\le & 2^{-i+1}+2^{-i}+\cdots +2^{-i-j+2}\\
& \le & 2^{-i+2}.\end{eqnarray*} Note that $(Y,d_Y)$ may not be
complete, and if not, the unique limit point is the desired limit
space $X$.

ii) Consider a subset of Cauchy sequences in $Y$,
$$\hat X=\{\{x_i\}: x_i\in X_i \text{ is a Cauchy sequence in
$Y$}\},$$ and define a pseudo-metric on $\hat X$,
$$\hat d(\{x_i\},\{y_i\})=\lim_{i\to \infty} d_Y(x_i,y_i),$$
where the existence of the limit is from
$$|d_Y(x_i,y_i)-d_Y(x_j,y_j)|\le d_Y(x_i,x_j)+d_Y(y_i,y_j)\to 0
\qquad \text{ as } i, j \to \infty.$$ Then $\hat d$ yields a metric
on the quotient space $X=\hat X/\sim$, where
$$\{x_i\}\sim \{y_i\} \qquad \text{ iff } \qquad \hat d(\{x_i\},
\{y_i\})=0.$$ We now define an admissible metric on $X\coprod Y$ by
declaring
$$d(\{x_i\},y)=\lim_{i\to \infty}d_Y(x_i,y).$$
(Because  $|d_Y(x_i,y)-d_Y(x_j,y)|\le d_Y(x_i,x_j)$, $d_Y(x_i,y)$ is
a Canchy sequence.) Since $d(\{x_i\},y)\ge d_Y(x_k,y)$ for some
$x_{k}\in \{x_i\}$, $d$ is indeed a metric (because
$d_{Y}(x_{k+j},y)>d_{Y}(x_k,y)>0$).

iii) Given $\epsilon>0$, we will construct a finite $\epsilon$-dense
subset of $X$ as follows: choose $i$ so that $2^{-i}<\frac\epsilon
5$. Because $X_i$ is compact, we may assume a finite $\frac \epsilon
5$-net, $\{x_i^1,...,x_i^\ell\}$, of $X_i$. Let
$x_{i+1}^1,...,x_{i+1}^\ell\in X_{i+1}$ such that
$d(x_i^j,x_{i+1}^j)<2^{-i}$. Let $x_{i+2}^1,...,x_{i+2}^\ell\in
X_{i+2}$ such that $d(x_{i+1}^j,x_{i+2}^j)<2^{-i-1}$. Repeating
this, we obtain, for each $k$, $x_{i+k}^1,...,x_{i+k}^\ell\in
X_{i+k}$ such that $d(x_{i+k-1}^j,x_{i+k}^j)<2^{-i-k+1}$. For each
$1\le j\le \ell$, it is clear that $\{x_{i+k}^j\}_{k=1}^\infty$ is a
Cauchy sequence, that is, $\{x_{i+k}^j\}_{k=1}^\infty\in X$.
Moreover, for each $1\le k<\infty$, $x_{i+k}^1,...,x_{i+k}^\ell$ is
$\frac{3\epsilon}5$-dense in $X_{i+k}$. This is because for any
$x\in X_{i+k}$, we can choose $x'\in X_i$ such that
$d(x,x')<2^{-i}$, and let $x_i^j\in \{x^j_i\}$ such that
$d(x',x_i^j)<\frac \epsilon 5$. Then $d(x,x_{i+k}^j)\le
d(x,x')+d(x',x_i^j)+d(x_i^j, x_{i+k}^j)<\frac \epsilon 5+2\cdot
2^{-i}<\frac{3\epsilon}5$.

Finally, we check that $\{\{x_{i+k}^1\}_{k=1}^\infty,...,
\{x_{i+k}^\ell\}_{k=1}^\infty\}$ is an $\epsilon$-dense subset in
$X$. Given any $\{y_k\}\in X$, we may assume that for a large $k$, $
d(y_k,y_{k+j})<\frac \epsilon 5$. Since $\{x_k^1,...,x_k^\ell\}$ is
a $\frac {3\epsilon} 5$-net for $X_k$, we may assume that
$d(y_k,x_k^s)<\frac {3\epsilon}5$. Then
$$d(y_{k+j},x^s_{k+j})\le d(y_k,y_{k+j})+d(y_k,x_k^s)+
d(x^s_k,x^s_{k+j})<\frac \epsilon 5+\frac {3\epsilon} 5 +\frac
\epsilon 5=\epsilon,$$ and thus $d(\{y_k\},\{x_k^s\})<\epsilon$.

iv) We shall show that for any $\epsilon>0 $,
$B_\epsilon(X)\supseteq X_i$ and $B_\epsilon(X_i)\supseteq X$ for
all large $i$.

For any $\epsilon>0$, let $2^{-i+1}<\epsilon$. For $x_i\in X_i$,
from the condition that $d_{i,i+j,H}(X_i,X_{i+j})<2^{-i+1}$, we
define a sequence $y_k\in X_k$ such that $d(y_k,y_{k+j})<2^{-k+1}$
and $y_i=x_i$ (we can choose $y_1, ...,y_{i-1}$ arbitrarily).
Clearly, $\{y_k\}$ is a Cauchy sequence and
$d(x_i,\{y_k\})<2^{-i+1}< \epsilon$. This show that $X_i\subseteq
B_\epsilon(X)$ for $i\ge \frac{-\ln \epsilon}{\ln 2}+1$.

For any $\{x_i\}\in X$, for $i$ large, we can assume that
$d(x_i,\{x_j\})<\epsilon$. Note that this does not give
$B_\epsilon(X_i)\supseteq X$, because how large $i$ is may depend on
$\{x_i\}$ in $X$. To overcome this trouble, by iii), we can assume a
finite $\frac \epsilon 4$-dense subset,
$\{y_i^1\}_{i=1}^\infty,...,\{y^\ell_i\}_{i=1}^ \infty$, for $X$.
For each $1\le j\le \ell$, we may assume some $N_j$ such that for
$i\ge N_j$, $d(y_i^j,\{y_i^j\})<\frac \epsilon 4$ and
$2^{-i+1}<\frac \epsilon 4$. Let $N=\max \{N_1,...,N_\ell\}$. For
any $\{x_i\}\in X$, we may assume some $1\le j\le \ell$ such that
$d(\{x_i\},\{y_i^j\}) <\frac \epsilon 4$. From the above, for each
$i\ge N$,
$$ d(y^j_i,\{x_i\})\le d(y^j_i,\{y^j_i\})+d(\{y^j_i\},\{x_i\})<\epsilon,$$
and thus $X\subseteq B_\epsilon(X_i)$. \end{proof}

\vskip2mm

A direct consequence of Proposition \ref{A1} is that   $(
\mathfrak{X},d_{GH})$ is a complete metric space.

A by-product of the above proof is that an abstract convergent
sequence, $X_i\stackrel{d_{GH}}\longrightarrow  X$, can be realized
as a concrete Hausdorff convergence, $d_H(X_i,X)\to 0$, in $\coprod
X_i\coprod X$ with an admissible metric $d$. In particular, it makes
sense to say that $x_i\in X_i, x_i\to x\in X$ because $d(x,x_i)\to
0$.

\vspace{0.7cm}

\section{ Estimates for Volume Forms } $$\text{ by  MARK GROSS}
\footnote{Supported by  NSF Grant  DMS-0805328 and DMS-0854987.
\\ Address: Mathematics Department, University of California San
Disgo, 9500 Gilman Drive, La Jolla, CA 92093-0112,
 USA.
   E-mail address:  mgross@math.ucsd.edu}$$

\vskip2mm

\begin{theorem}\label{B.1}  Let $\pi: \mathcal{M}\rightarrow \Delta$ be a flat
projective family of $n$-dimensional Calabi-Yau varieties,
with $M_t=\pi^{-1}(t)$
non-singular for $t\not=0$ and $M_0=\pi^{-1}(0)$ a variety with
canonical singularities. After embedding the family $\mathcal{M}$ in
$\mathbb{CP}^N\times \Delta$, let $\omega_t$ denote the restriction
of the Fubini-Study metric on $\mathbb{CP}^N$ to $M_t$. Furthermore,
let $\Omega$ be a nowhere vanishing holomorphic section of the
relative canonical bundle $\mathcal{K}_{\mathcal{M}/\Delta}$, and set
$\Omega_t =\Omega|_{M_t}$. Then \begin{itemize}
  \item[i)]There is a $\kappa$ independent of
$t$ such that
\[
(-1)^{\frac{n^{2}}{2}}\Omega_t\wedge\bar\Omega_t>\kappa \omega_t^n.
\]
\item[ii)] There is a constant $\Lambda$ independent of $t$ such that
\[
(-1)^{\frac{n^{2}}{2}} \int_{M_t}\Omega_t\wedge\bar\Omega_t<\Lambda.
\]
\end{itemize}
\end{theorem}

\proof For i), we use an argument similar to that in \cite{EGZ},
Lemma 6.4. Let $\mathcal{M}^{sm}$ denote the set of points of
$\mathcal{M}$ where $\pi$ is smooth, i.e., the set of points where
$\mathcal{M}$ is non-singular and $\pi_*$ is surjective. Let
$p\in\mathcal{M}$ be a point, and consider an open neighbourhood
$U_p$ of $p$ which embeds into $\mathbb{C}^{N+1}$ via
$\iota:U_p\rightarrow \mathbb{C}^{N+1}$, with coordinates
$t,z_1,\ldots,z_N$. The Fubini-Study form is comparable to
$\omega=\sqrt{-1}\sum\limits_{i=1}^N dz_i\wedge d\bar z_i$, so we
can assume that locally $\omega_t$ is the restriction of $\omega$ to
$U_p\cap M_t$. Now $\omega^n=n!(-1)^{\frac{n }{2}}\sum\limits_I
dz_I\wedge d\bar z_I$, where the sum is over all index sets
$I\subseteq\{1,\ldots,N\}$ with $\#I=n$. Now as $\iota^*(dz_I)$ is a
relative holomorphic $n$-form on $U_p\cap \mathcal{M}^{sm}$, there
is a holomorphic function $f_I$ on $U_p\cap \mathcal{M}^{sm}$ such
that $\iota^*(dz_I)=f_I\Omega$. Note that since $\mathcal{M}$ is
necessarily normal  and $\mathcal{M}\setminus\mathcal{M}^{sm}$ is
codimension $\ge 2$, we can apply Hartog's theorem for normal
analytic spaces to extend $f_I$ to a holomorphic function on $U_p$.
Thus
\[
\iota^*\omega^n=C (-1)^{\frac{n^{2} }{2}} \left(\sum_I
|f_I|^2\right)\Omega\wedge\bar\Omega.
\]
On an open neighbourhood $V_p\subset\subset U_p$ of $p$, $|f_I|$ is
bounded. This gives the desired result.

For ii), we need to apply some standard results from Hodge theory.
After making a base-change $\Delta\rightarrow\Delta$ given by
$t\mapsto t^k$ for some $k$, we can assume that the monodromy
operator $T$ about the origin acting on $H^n(M_{t_0},\mathbb{C})$ is
unipotent, i.e., $(T-I)^m=0$ for some $m$. Here
$t_0\in\Delta^*=\Delta\setminus\{0\}$ is a fixed basepoint. Let
\[
N=\log(T-I);
\]
this makes sense via the power series expansion. By the stable
reduction theorem \cite{KKMS}, one has a diagram
\[
\xymatrix@C=30pt {
\widetilde{\mathcal{M}} \ar[r]^{\eta}\ar[rd]_{\tilde{\pi}}& \mathcal{M} \ar[d]^\pi\\
&\Delta }
\]
in which $\eta$ is an isomorphism outside the central fibre and
$\tilde{\pi}$ is normal crossings, i.e., locally around points of
$\widetilde{M}_0=\tilde{\pi}^{-1}(0)$ there are coordinates
$z_1,\ldots,z_{n+1}$ on $\widetilde{\mathcal{M}}$ such that
$t=z_1\cdots z_p$ for some $p\le n+1$. One has the sheaf
$\Omega^1_{\widetilde{\mathcal{M}}}(\log \widetilde M_0)$ of
logarithmic differentials on $\widetilde{\mathcal{M}}$ locally
generated by ${dz_1\over z_1},\ldots,{dz_p\over
z_p},dz_{p+1},\ldots, dz_{n+1}$, and the sheaf of relative
logarithmic diffentials
$\Omega^1_{\widetilde{\mathcal{M}}/\Delta}(\log\widetilde{M}_0)$ is
obtained by dividing out by the relation ${dt\over t}=0$. It is
standard (see for example the book \cite{PS} for the full background
used here) that
$\Omega^1_{\widetilde{\mathcal{M}}/\Delta}(\log\widetilde{M}_0)$ is
a rank $n$ vector bundle, and if $X$ is an irreducible component of
$\widetilde{M}_0$, then
$\Omega^1_{\widetilde{\mathcal{M}}/\Delta}(\log\widetilde{M}_0)|_X=\Omega^1_X(\log\partial
X)$, where $\partial X=X\cap \widetilde{S}$, and $\widetilde{S}$ is
the singular set of $\widetilde{M}_0$. One then obtains the
logarithmic de Rham complex
$\Omega^{\bullet}_{\widetilde{\mathcal{M}}/\Delta}(\log\widetilde{M}_0)$,
with
$\Omega^p_{\widetilde{\mathcal{M}}/\Delta}(\log\widetilde{M}_0)$ the
$p$-th exterior power of the sheaf of relative log differentials,
and the differential $d$ is the ordinary exterior derivative. In
particular, we have the line bundle
$\Omega^n_{\widetilde{\mathcal{M}}/\Delta}(\log\widetilde{M}_0)$.

By \cite{Steenbrink}, Theorem 2.11,
$\tilde{\pi}_*\Omega^n_{\widetilde{\mathcal{M}}/\Delta}(\log\widetilde{M}_0)$
is a vector bundle whose fibre over $t\not=0$ is
$H^0(M_t,\mathcal{K}_{M_t})$. Hence this is a line bundle. On the
other hand, by assumption on $\mathcal{M}$,
$\mathcal{K}_{\mathcal{M}/\Delta}\cong\mathcal{O}_\mathcal{M}$ and
so $\pi_*\mathcal{K}_{\mathcal{M}/\Delta}$ is also a line bundle.
Let $\mathcal{M}^o$ be the largest open set so that
$\eta^{-1}(\mathcal{M}^o)\rightarrow\mathcal{M}^o$ is an
isomorphism, and let $i:\mathcal{M}^o\rightarrow\mathcal{M}$ be the
inclusion. Then the codimension of
$\mathcal{M}\setminus\mathcal{M}^o$ in $\mathcal{M}$ is at least
two. Since $\mathcal{M}\setminus\mathcal{M}^{sm}$ has codimension at
least two, $\mathcal{M}\setminus
(\mathcal{M}^{sm}\cap\mathcal{M}^o)$ has codimension at least two.
We have a composition of canonical sheaf homomorphisms
\[
\eta_*\Omega^n_{\widetilde{\mathcal{M}}/\Delta}(\log\widetilde{M}_0)\rightarrow
i_*i^*\eta_*\Omega^n_{\widetilde{\mathcal{M}}/\Delta}(\log\widetilde{M}_0)\cong
\mathcal{K}_{\mathcal{M}/\Delta},
\]
the latter isomorphism by Hartog's theorem and the fact that the
isomorphism holds over $\mathcal{M}^{sm}\cap\mathcal{M}^o$. Applying
$\pi_*$ then gives a map
\begin{equation}
\label{linebundlemap} \tilde{\pi}_*
\Omega^n_{\widetilde{\mathcal{M}}/\Delta}(\log\widetilde{M}_0)\rightarrow
\pi_*\mathcal{K}_{\mathcal{M}/\Delta}.
\end{equation}
This map is an isomorphism over $\Delta^*$, and hence is necessarily
an inclusion of sheaves. To show it is in fact an isomorphism, we
need to show that any section of
$\mathcal{K}_{\mathcal{M}/\Delta}|_{M_0}=\mathcal{K}_{M_0}$ comes
from a section of
$\Omega^n_{\widetilde{\mathcal{M}}/\Delta}(\log\widetilde{M}_0)|_{\widetilde{M}_0}$.
To see this, let $X_0$ be the proper transform of $M_0$ in
$\widetilde{M}_0$. Then $\eta_0:X_0\rightarrow M_0$ is a resolution
of singularities, and since $M_0$ has canonical singularities, we
have
\[
\mathcal{K}_{X_0}=\eta_0^*\mathcal{K}_{M_0}+\sum_E a_EE,
\]
where the sum is over all exceptional divisors $E$ of $\eta_0$ and
$a_E\ge 0$. (Note $a_E$ is an integer since $M_0$ is Gorenstein). On
the other hand, $\Omega^n_{X_0}(\log \partial X_0)$ is
$\mathcal{K}_{X_0}+\sum\limits_E E$, where the sum is again over all
exceptional divisors of $\eta_0$. So $\eta_0^*\Omega_0$, viewed as a
section of $\Omega^n_{X_0}(\log\partial X_0)$, has a zero of order
at least $1$ along each exceptional divisor $E$. Thus
$\eta_0^*\Omega_0$ extends by zero to a section of
$\Omega^n_{\widetilde{\mathcal{M}}/\Delta}(\log\widetilde{M}_0)|_{\widetilde
{M}_0}$. Thus \eqref{linebundlemap} is surjective, hence an
isomorphism.

We now recall some standard material concerning the limiting mixed
Hodge structure and the nilpotent orbit theorem. Denote by
$\mathcal{H}^n$ the vector bundle $\mathbb{R}^n\tilde{
\pi}_*\Omega^{\bullet}_{\widetilde{\mathcal{M}}/\Delta}
(\log\widetilde{M}_0)$. The fibre of this bundle at $t$ is
isomorphic to $H^n(M_t,\mathbb{C})$. This bundle comes along with
the Gauss-Manin connection, which is flat with a regular singular
point at $0\in\Delta$.

Let $j:H\rightarrow \Delta^*$ be the universal cover, with $H$ the
upper half-plane, with coordinate $w={1\over 2\pi \sqrt{-1}}\log t$.
Then $j^*\mathcal{H}^n$ is now canonically identified with the
trivial bundle $H\times H^n(M_{t_0},\mathbb{C})$ via parallel
transport by the Gauss-Manin connection. If $e\in
H^n(M_{t_0},\mathbb{C})$, one obtains a constant section $\sigma_e$
of $j^*\mathcal{H}^n$ by parallel transport, and then
$e^{-wN}\sigma_e$ descends to a single-valued section of
$\mathcal{H}^n$ over $\Delta^*$. The bundle $\mathcal{H}^n$ is then
the canonical extension of $\mathcal{H}^n|_{\Delta^*}$, i.e., the
extension in which, for a basis $e_1,\ldots,e_s$ of
$H^n(M_{t_0},\mathbb{C})$, $e^{-wN}\sigma_{e_1},\ldots
e^{-wN}\sigma_{e_n}$ form a holomorphic frame. In particular, there
is an isomorphism of the fibre
$\mathcal{H}^n_0=\mathbb{H}^n(\widetilde{M}_0,
\Omega^{\bullet}_{\widetilde{\mathcal{M}}/\Delta}(\log\widetilde{M}_0)|_{\widetilde{M}_0})$
with $H^n(M_{t_0},\mathbb{C})$, isomorphic to the space of flat
sections of $j^*\mathcal{H}^n$.

We also have an inclusion
\[
\mathcal{F}^n:=\tilde
\pi_*\Omega^n_{\widetilde{\mathcal{M}}/\Delta}(\log\widetilde{M}_0)
\hookrightarrow \mathcal{H}^n.
\]
The fibre of $\mathcal{F}^n$ over $0\in\Delta$ is
$\mathcal{F}^n_{\lim}\subseteq H^n(M_{t_0},\mathbb{C})$ under the
above isomorphism, a piece of the limiting mixed Hodge structure. In
particular, the value of the holomorphic section $\Omega$ of
$\pi_*\mathcal{K}_{\mathcal{M}/\Delta}$ at $0$ under the isomorphism
\eqref{linebundlemap} defines a class $\Omega_{\lim}\in
\mathcal{F}^n_{\lim}$.

We now apply the nilpotent orbit theorem (see e.g.,
\cite{Griffiths}, Chapter IV, for an exposition of this material).
Let $\phi:H \rightarrow \mathbb{P}(H^n(M_{t_0},\mathbb{C}))$ be the
period map, with, for $w\in H$, $\phi(w)$ being the one-dimensional
subspace $(j^*\mathcal{F}^n)_w \subseteq (j^*\mathcal{H}^n)_w\cong
H^n(M_{t_0},\mathbb{C})$, the latter identification via the
Gauss-Manin connection. Then $e^{-wN}\phi:H\rightarrow
\mathbb{P}(H^n(M_{t_0},\mathbb{C}))$ descends to a map
$\psi:\Delta^* \rightarrow \mathbb{P}(H^n(M_{t_0},\mathbb{C}))$
which in turn extends across the origin, with
$\psi(0)=[\Omega_{\lim}]$. The nilpotent orbit is the map
$\phi^{nil}:H\rightarrow \mathbb{P}(H^n(M_{t_0},\mathbb{C}))$ given
by
\[
\phi^{nil}(w)=e^{wN}\psi(0)=e^{wN}[\Omega_{\lim}].
\]
The nilpotent orbit theorem states that $\phi^{nil}$ is a good
approximation to $\phi$, i.e., with a suitable metric on
$\mathbb{P}(H^n(M_{t_0},\mathbb{C}))$ inducing a distance function
$\rho$, we have constants $A$ and $B$ such that for ${\rm Im} w\ge
A>0$,
\[
\rho(\phi(w),\phi^{nil}(w))\le ({\rm Im}  w)^B e^{-2\pi{\rm Im} w}.
\]
This implies that $\int_{M_t}\Omega_t\wedge\bar\Omega_t$ is bounded
independently of $t$ near $0$ provided that $\int_{M_{t_0}}
e^{wN}\Omega_{\lim}\wedge \overline{e^{wN}\Omega_{\lim}}$ is bounded
for ${\rm Im}  w\ge A$.

Now we apply the argument of Proposition 2.3 and Theorem 2.1 of
\cite{Wang}. The argument of Proposition 2.3 tells us that
$\widetilde{M}_0$ has an irreducible component (in fact $X_0$) with
$H^{n,0}(X_0,\mathbb{C})\not=0$. Thus, by the first line of the
proof of Theorem 2.1, $N\mathcal{F}^n_{\infty}=0$. So in particular,
$e^{wN}\Omega_{\lim}=\Omega_{\lim}$, giving the desired boundedness.
\qed

\end{document}